\documentclass[a4paper,12pt]{article}     

\usepackage{graphicx}		  
\usepackage{amsmath,amssymb}	

\newtheorem{lemma}{Lemma}[section]
\newtheorem{theorem}[lemma]{Theorem}
\newtheorem{proposition}[lemma]{Proposition}

\newtheorem{assumption}{Assumption}

\newenvironment{proof}{
\hspace*{-9mm}
{ \it Proof.}}
{\hfill {$\square$}\vspace{1.5em}}

\begin{document}

\begin{center}{
{\Large 
Properties of minimal charts and their applications IV: Loops}
\vspace{10pt}
\\ 
Teruo NAGASE and Akiko SHIMA\footnote
{
The second author is partially supported by Grant-in-Aid for Scientific Research (No.23540107), Ministry of Education, Science and Culture, Japan. 
}
}
\end{center}

\begin{abstract}
We investigate minimal charts with loops, a simple closed curve consisting of edges of label $m$ containing exactly one white vertex. We shall show that there does not exist any loop in a minimal chart with exactly seven white vertices in this paper.
\end{abstract}

\setcounter{section}{0}
\section{{\large Introduction}}


Charts are oriented labeled graphs in a disk which represent surface braids (see  \cite{KnottedSurfaces},\cite{BraidBook}, and see Section 2 for the precise definition of charts, see \cite[Chapter 14]{BraidBook} for the definition of surface braids).
In a chart there are three kinds of vertices; white vertices, crossings and black vertices. 
A C-move is a local modification 
of charts in a disk.
The closures of surface braids are embedded closed oriented surfaces in 4-space ${\Bbb R}^4$
 (see \cite[Chapter 23]{BraidBook} for the definition of the closures of surface braids). 
A C-move between two charts induces an ambient isotopy between the closures of the corresponding two surface braids.

We will work in the PL category or smooth category. All submanifolds are assumed to be locally flat.
In \cite{ONS},
we showed that there is no minimal chart with exactly five vertices
 (see Section 2 for the precise definition of minimal charts). 
Hasegawa proved that there exists a minimal chart with exactly
six white vertices \cite{H1}. 
This chart represents the surface braid whose
closure is ambient isotopic to a 2-twist spun trefoil.
In \cite{INS} and \cite{NST},
we investigated minimal charts with exactly four white vertices.

A {\it loop} of label $m$ in a chart is a simple closed curve consisting of edges of label $m$ which contains exactly one white vertex. 

In this paper, 
we investigate properties of minimal charts and
need to prove that
there is no minimal chart with exactly seven white vertices.
In particular we investigate a minimal chart containing a loop which bounds a disk containing three white vertices in its interior (see Lemma~\ref{CorThreeRed}).
We shall show the following theorem:

\begin{theorem}
\label{NoLoop}
There does not exist any loop in a minimal chart with exactly seven white vertices.
\end{theorem}

Let $\Gamma$ be a chart. If an object consists of some edges of $\Gamma$, arcs in edges of $\Gamma$ and arcs around white vertices,
then the object is called {\it a pseudo chart}.

The paper is organized as follows.
In Section~\ref{s:Prel},
we define charts and minimal charts.
In Section~\ref{s:kAngledDisk},
we give pseudo charts in a disk bounded by a loop in a minimal chart
(Lemma~\ref{LoopTwoVertices} and Lemma~\ref{CorThreeRed}).
And we review a $k$-angled disk, a disk whose boundary consists of edges of label $m$
and contains exactly $k$ white vertices.
In Section~\ref{s:UsefulLemma},
we give useful lemmata
and introduce notations.
In Section~\ref{s:ProofLemma},
we give a proof of Lemma~\ref{CorThreeRed}.
In Section~\ref{s:SolarEclipse},
we discuss about a solar eclipse, two loops with exactly one white vertex.
In Section~\ref{s:Eyeglasses},
we discuss about two kinds of subgraphs containing two loops of label $m$, 
one called
a pair of eyeglasses and
the other called a pair of skew eyeglasses.
In Section~\ref{s:TriangleLemma},
we prove Triangle Lemma(Lemma~\ref{LemmaTriangle}) for a 3-angled disk.
In Section~\ref{s:ProofTheorem},
we prove Theorem~\ref{NoLoop}.

\section{{\large Preliminaries}}
\label{s:Prel}

Let $n$ be a positive integer. An {\it $n$-chart} is an oriented labeled graph in a disk,
which may be empty or have closed edges without vertices, called {\it hoops},
satisfying the following four conditions:
\begin{enumerate}
	\item[(i)] Every vertex has degree $1$, $4$, or $6$.
	\item[(ii)] The labels of edges are in $\{1,2,\dots,n-1\}$.
	\item[(iii)] In a small neighborhood of each vertex of degree $6$,
	there are six short arcs, three consecutive arcs are oriented inward and the other three are outward, and these six are labeled $i$ and $i+1$ alternately for some $i$,
	where the orientation and the label of each arc are inherited from the edge containing the arc.
	\item[(iv)] For each vertex of degree $4$, diagonal edges have the same label and are oriented coherently, and the labels $i$ and $j$ of the diagonals satisfy $|i-j|>1$.
\end{enumerate}
We call a vertex of degree $1$ a {\it black vertex,} a vertex of degree $4$ a {\it crossing}, and a vertex of degree $6$ a {\it white vertex} respectively (see Fig.~\ref{fig01}).
Among six short arcs
in a small neighborhood of
a white vertex,
a central arc of each three consecutive arcs
oriented inward or outward 
is called 
a {\it middle arc} at the white vertex
(see Fig.~\ref{fig01}(c)).
There are two middle arcs
in a small neighborhood of
each white vertex.

\begin{figure}[t]
\centerline{\includegraphics{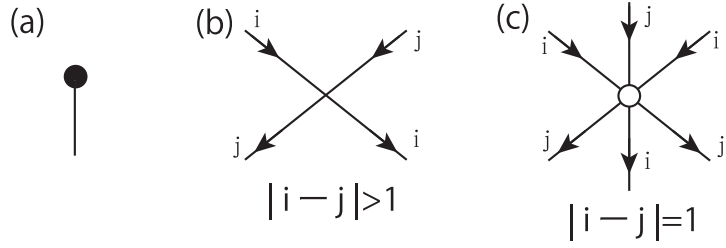}}
\vspace{5mm}
\caption{\label{fig01}(a) A black vertex. (b) A crossing. (c) A white vertex.}
\end{figure}

Now {\it $C$-moves} are local modifications 
of charts in a disk
as shown in 
Fig.~\ref{fig02}
(see \cite{KnottedSurfaces}, 
\cite{BraidBook}, \cite{Tanaka} 
for the precise definition).
These C-moves as shown in 
Fig.~\ref{fig02} 
are examples of C-moves.


\begin{figure}[t]
\centerline{\includegraphics{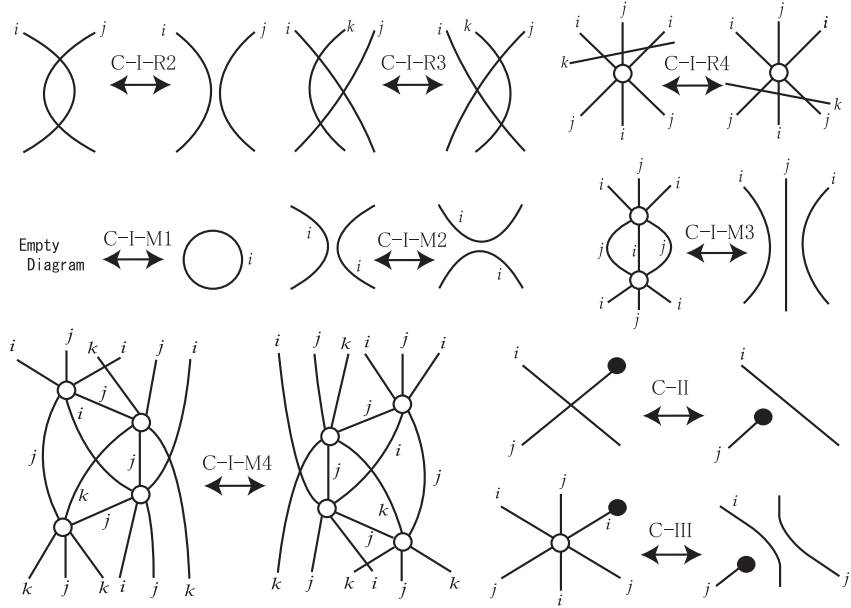}}
\vspace{5mm}
\caption{\label{fig02}For the C-III move, the edge containing the black vertex does not contain a middle arc in the left figure.}
\end{figure}

Let $\Gamma$ be a chart. For each label $m$, we denote by $\Gamma_m$ the 'subgraph' of $\Gamma$ consisting of all the edges of label $m$ and their vertices.
An edge of $\Gamma$ is the closure of a connected component of the set obtained by taking out all white vertices and crossings from $\Gamma$.
On the other hand, we assume that 
\begin{enumerate}
	\item[] an {\it edge} of $\Gamma_m$ is the closure of a connected component of the set obtained by taking out all white vertices from $\Gamma_m$.
\end{enumerate}
Thus 
any vertex of $\Gamma_m$ is a black vertex or a white vertex.
Hence any crossing of $\Gamma$ is not considered as a vertex of $\Gamma_m$.

Let $\Gamma$ be a chart, $m$ a label of $\Gamma$, and $e$ an edge of $\Gamma$ or $\Gamma_m$.
The edge $e$ is called 
a {\it free edge}
if it has
two black vertices.
The edge $e$ is called 
a {\it terminal edge}
if it has
a white vertex and a black vertex.
The edge $e$ is a {\it loop} 
if it is a closed edge with only one white vertex.
Note that 
free edges, terminal edges, and loops 
may contain crossings of $\Gamma$.

Two charts are said to be {\it C-move equivalent} 
if there exists
a finite sequence of C-moves 
which modifies one of the two charts 
to the other.

For each chart $\Gamma$,
let $w(\Gamma)$ and $f(\Gamma)$ be the number of white vertices, and the number of free edges respectively.
The pair $(w(\Gamma), -f(\Gamma))$ is called a {\it complexity} of the chart (see \cite{BraidThree}).
A chart $\Gamma$ is called a {\it minimal chart} if its complexity is minimal among the charts C-move equivalent to the chart $\Gamma$ with respect to the lexicographic order of pairs of integers.

We showed the difference of a chart in a disk and in a 2-sphere (see \cite[Lemma 2.1]{ChartAppl}).
This lemma follows from that there exists a natural one-to-one correspondence between $\{$charts in $S^2\}/$C-moves and $\{$charts in $D^2\}/$C-moves, conjugations
(\cite[Chapter 23 and Chapter 25]{BraidBook}).
To make the argument simple, we assume that 
the charts lie on the 2-sphere instead of the disk.
\begin{assumption}
In this paper,
all charts are contained in the $2$-sphere $S^2$.
\end{assumption}
We have the special point in the 2-sphere $S^2$, called the point at infinity,
 denoted by $\infty$.
In this paper, all charts are contained in a disk such that the disk 
does not contain the point at infinity $\infty$.


Let $\Gamma$ be a chart,
and $m$ a label of $\Gamma$. 
A {\it hoop} is a closed edge of $\Gamma$ without vertices 
(hence without crossings, neither).
A {\it ring} is a closed edge of $\Gamma_m$ containing a crossing but not containing any white vertices.
A hoop is said to be {\it simple} 
if one of the two complementary domains
of the hoop
does not contain any white vertices.

We can assume that
all minimal charts $\Gamma$
satisfy the following five conditions 
(see \cite{ChartAppl},\cite{ChartAppII},\cite{ChartAppIII}):

\begin{assumption}
\label{NoTerminal}
No terminal edge of $\Gamma_m$
contains a crossing.
Hence 
any terminal edge of $\Gamma_m$
contains a middle arc.
\end{assumption}

\begin{assumption}
 No free edge of $\Gamma_m$ contains a crossing. Hence any free edge
of $\Gamma_m$ is a free edge of $\Gamma$.
\end{assumption}

\begin{assumption}
\label{NoSimpleHoop}
All free edges and simple hoops in $\Gamma$ 
are moved into a small neighborhood $U_\infty$ 
of the point at infinity $\infty$. 
Hence
we assume that 
$\Gamma$ does not contain free edges
nor simple hoops, 
otherwise mentioned. 
\end{assumption}

\begin{assumption}
\label{AssumptionRing}
Each complementary domain of
any ring and hoop must contain 
at least one white vertex. 
\end{assumption}

\begin{assumption}
The point at infinity $\infty$ is moved in any complementary domain of $\Gamma$.
\end{assumption}

In this paper
for a set $X$ in a space
we denote 
the interior of $X$,
the boundary of $X$ and
the closure of $X$
by Int$X$, $\partial X$
and $Cl(X)$
respectively.


\section{{\large $k$-angled disks}}
\label{s:kAngledDisk}

Let $\ell$ be a loop of label $m$ in a chart $\Gamma$. 
Let $e$ be the edge of $\Gamma_m$ containing the white vertex in $\ell$ with $e\not=\ell$. 
Then the loop $\ell$ bounds two disks on the 2-sphere. 
One of the two disks does not contain the edge $e$. 
The disk is called {\it the associated disk of the loop $\ell$} 
(see Fig.~\ref{fig03}).


\begin{figure}[t]
\centerline{\includegraphics{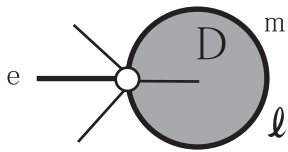}}
\vspace{5mm}
\caption{\label{fig03}The gray region is the associated disk of a loop $\ell$.}
\end{figure}

Let $\Gamma$ be a chart, $m$ a label of $\Gamma$, $D$ a disk, and $k$ a positive integer.
If $\partial D$ consists of $k$ edges of  $\Gamma_m$, 
then $D$ is called {\it a $k$-angled disk of $\Gamma_m$}. Note that the boundary $\partial D$ may contain crossings, and
each of two disks bounded by a loop of label $m$ is a $1$-angled disk of $\Gamma_m$.

Let $\Gamma $ and $\Gamma^\prime $ be C-move equivalent charts. 
Suppose that a pseudo chart $X$ of $\Gamma$ is also a pseudo chart of $\Gamma^\prime$. 
Then we say that 
$\Gamma$ is modified to $\Gamma^\prime$ by {\it C-moves keeping $X$ fixed}.
In Fig.~\ref{fig04},
we give examples of C-moves keeping pseudo charts  fixed.

\begin{figure}[t]
\centerline{\includegraphics{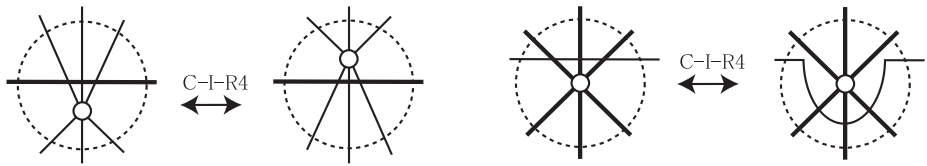}}
\vspace{5mm}
\caption{\label{fig04} C-moves keeping thicken figures fixed.}
\end{figure}

Let $\Gamma$ be a chart, and $m$ a label of $\Gamma$.
Let $D$ be a $k$-angled disk of $\Gamma_m$, 
and 
$G$ a pseudo chart in $D$ with $\partial D\subset G$.
Let $r:D\to D$ be a reflection of $D$, and $G^*$ the pseudo chart obtained from $G$ by changing the orientations of all of the edges.
Then the set $\{G,G^*, r(G), r(G^*)\}$ 
is called the {\it RO-family of the pseudo chart $G$}.

{\it Warning.}
To draw a pseudo chart in a RO-family,
we draw a part of $\Gamma\cap N$
here $N$ is a regular neighborhood of a $k$-angled disk $D$.
For example,
we draw the pseudo chart in Fig.~\ref{fig05}(a)
for a pseudo chart in a RO-family in Fig.~\ref{fig05}(b).

Let $X$ be a set in a chart $\Gamma$.
Let
 $$w(X)=\text{the number of white vertices in $X$,}$$
$$c(X)=\text{the number of crossings in }X.$$

\begin{lemma}
\label{LoopTwoVertices}
{\em $($\cite[Lemma 4.2 and Lemma 8.1]{ChartAppII}$)$}
Let $\Gamma$ be a minimal chart with a loop $\ell^*$ of label $m$ with the associated disk $D^*$.
Let $\varepsilon$ be the integer in $\{+1,-1\}$ such that the white vertex in $\ell^*$ is in $\Gamma_{m+\varepsilon}$.
Then we have the following:
\begin{enumerate}
\item[$(1)$] $w(\Gamma\cap${\rm Int}$D^*)\ge2$ and $w(\Gamma\cap(S^2-D^*))\ge2$.
\item[$(2)$] If $w(\Gamma\cap${\rm Int}$D^*)=2$,
then $D^*$ contains a pseudo chart of the RO-family of the pseudo chart as shown in Fig.~\ref{fig05}{\rm(a)}
by C-moves in $D^*$ keeping $\partial D^*$ fixed. 
\end{enumerate}
\end{lemma}

\begin{figure}[t]
\centerline{\includegraphics{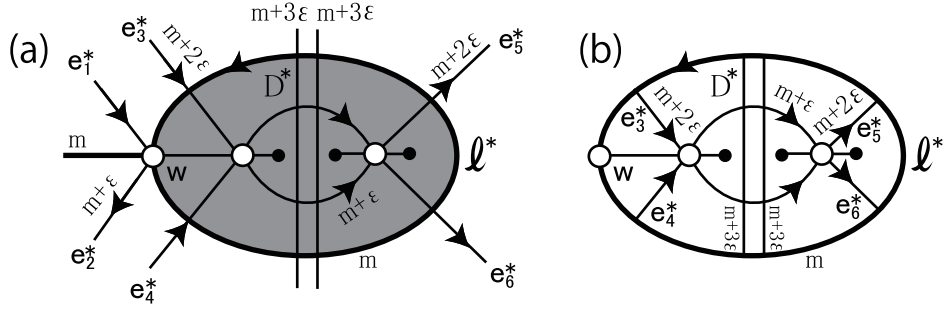}}
\vspace{5mm}
\caption{\label{fig05} The gray region is a disk $D^*$, the thick lines are edges of label $m$, and $\varepsilon\in\{+1,-1\}$}
\end{figure}

Let $\Gamma$ be a minimal chart,
and $m$ a label of $\Gamma$.
Let $D$ be a $k$-angled disk of $\Gamma_m$.
The pair of integers 
$(w(\Gamma\cap$Int$D),c(\Gamma\cap\partial D))$
is called the {\it local complexity 
with respect to $D$}.
Let ${\Bbb S}$ be the set of all
minimal charts obtained from $\Gamma$ by C-moves in a regular neighborhood of
$D$ keeping $\partial D$ fixed.
The chart $\Gamma$ is said to be 
{\it locally minimal
with respect to $D$}
if its local complexity
with respect to $D$
is minimal
among the charts in ${\Bbb S}$
with respect to 
the lexicographic order of pairs of integers.

The following lemma will be proved in Section~\ref{s:ProofLemma}.

\begin{lemma}
\label{CorThreeRed}
Let $\Gamma$ be a minimal chart with a loop $\ell$ of label $k$.
Let $\mu$ be the integer in $\{+1,-1\}$ 
such that the white vertex in $\ell$ is in $\Gamma_{k+\mu}$.
If $\Gamma$ is locally minimal with respect to the associated disk $D$ of $\ell$
and if $w(\Gamma\cap${\rm Int}$D)\le3$, 
then $D$ contains a pseudo chart of the RO-families of the five pseudo charts as shown in Fig.~\ref{fig05}{\rm (a)} and Fig.~\ref{fig06}.
\end{lemma}

\begin{figure}[t]
\centerline{\includegraphics{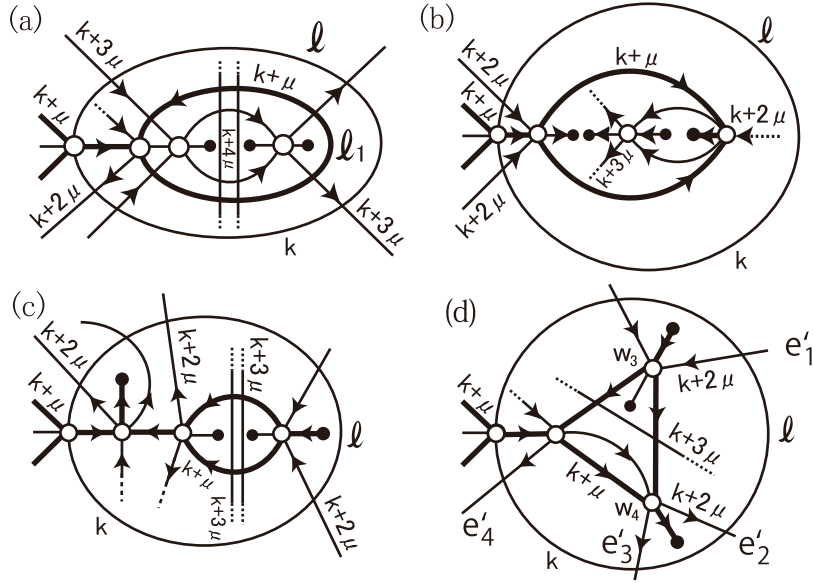}}
\vspace{5mm}
\caption{\label{fig06}The thick lines are edges of label $k+\mu$ where $\mu\in\{+1,-1\}$.}
\end{figure}

Let $\Gamma$ be a chart, and
$m$ a label of $\Gamma$.
An edge of $\Gamma_m$ is called a {\it feeler} of a $k$-angled disk $D$ of $\Gamma_m$
if the edge intersects $N-\partial D$
where $N$ is a regular neighborhood of $\partial D$ in $D$.

The following lemma will be used in the proofs of Theorem~\ref{NoLoop} and Lemma~\ref{CorThreeRed}.

\begin{lemma}
{\em $($\cite[Theorem 1.1]{ChartAppIII}$)$}
\label{Theorem2AngledDisk}
Let $\Gamma$ be a minimal chart.
Let $D$ be a $2$-angled disk of $\Gamma_m$ with at most one feeler
such that $\Gamma$ is locally minimal with respect to $D$.
If $w(\Gamma\cap${\rm Int}$D)\le1$,
then $D$ contains a pseudo chart of the RO-families of the five pseudo charts as shown in Fig.~\ref{fig07}. 
\end{lemma}

\begin{figure}[t]
\centerline{\includegraphics{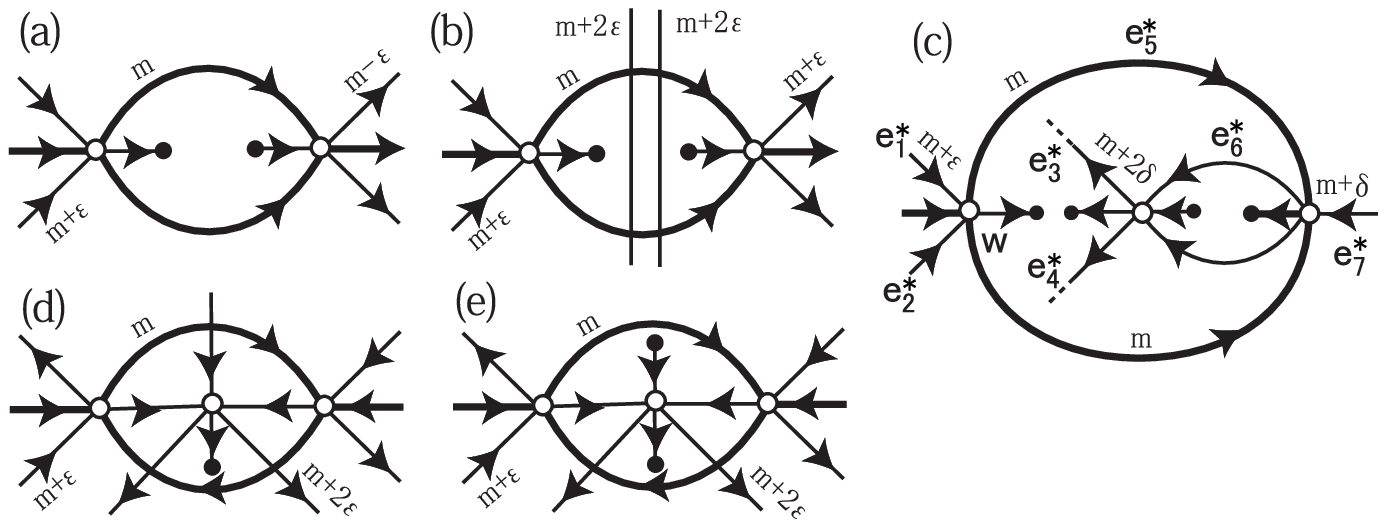}}
\vspace{5mm}
\caption{\label{fig07}
(a), (b), (d), (e) 2-angled disks without feelers. (c) A 2-angled disk with one feeler. 
The thick lines are edges of label $m$,
and $\varepsilon,\delta\in\{+1,-1\}$.}
\end{figure}

Let $\Gamma$ be a chart, 
and $D$ a $k$-angled disk of $\Gamma_m$. 
If any feeler of $D$
 is a terminal edge,
then we say that $D$ is {\it special}.

The following lemma will be used in Case (e) and Case (f) of the proof of Lemma~\ref{CorThreeRed}.

\begin{lemma}
{\em $($\cite[Theorem 1.2]{ChartAppIII}$)$}
\label{Theorem3AngledDisk}
Let $\Gamma$ be a minimal chart.
Let $D$ be a special $3$-angled disk of $\Gamma_m$
such that $\Gamma$ is locally minimal with respect to $D$.
If $w(\Gamma\cap${\rm Int}$D)=0$, 
then $D$ contains a pseudo chart of the RO-families of the two pseudo charts as shown in Fig.~\ref{fig08}. 
\end{lemma}

\begin{figure}[t]
\centerline{\includegraphics{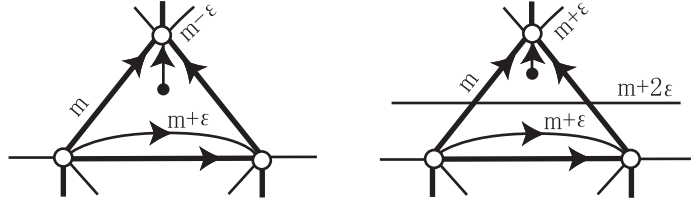}}
\vspace{5mm}
\caption{\label{fig08}The thick lines are edges of label $m$, and $\varepsilon\in\{+1,-1\}$.}
\end{figure}


\section{{\large Useful lemmata}}
\label{s:UsefulLemma}

Let $\Gamma$ be a chart. 
Let $D$ be a disk 
such that 
\begin{enumerate}
\item[(1)] $\partial D$ consists of an edge $e_1$ of $\Gamma_m$ and an edge $e_2$ of $\Gamma_{m+1}$, and 
\item[(2)] any edge containing a white vertex in $e_1$ does not intersect the open disk Int$D$.
\end{enumerate}
Note that $\partial D$ may contain crossings.
Let $w_1$ and $w_2$ be the white vertices in $e_1$. 
If the disk $D$ satisfies one of the following conditions, then $D$ is called  {\it a lens of type $(m,m+1)$}
(see Fig.~\ref{lens}):
\begin{enumerate}
	\item[(i)] Neither $e_1$ nor $e_2$ contains a middle arc. 
	\item[(ii)] One of the two edges $e_1$ and $e_2$ contains middle arcs at both white vertices $w_1$ and $w_2$ simultaneously.
\end{enumerate}

\begin{figure}[t]
\centerline{\includegraphics{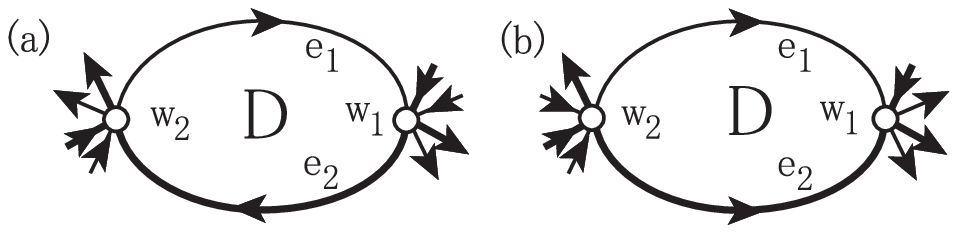}}
\vspace{5mm}
\caption{\label{lens}}
\end{figure}

The following lemma will be used in Case (a), Case (d) and Case (e) of the proof of Lemma~\ref{CorThreeRed}.

\begin{lemma}
{\em $($\cite[Theorem 1.1]{ChartAppl}$)$}
\label{PolandLemma}
Let $\Gamma$ be a minimal chart.
Then there exist at least three white vertices in the interior of any lens.
\end{lemma}

\begin{lemma}
\label{ComponentTwoWhite}
{\em $($\cite[Lemma 6.1]{ChartAppII}$)$}
Let $\Gamma$ be a minimal chart.
If a connected component $G$ of $\Gamma_m$
contains a white vertex,
then $G$ contains at least two white vertices.
\end{lemma}

Let $\alpha$ be an arc,  
and $p,q$ points in $\alpha$. 
We denote by $\alpha[p,q]$ the subarc of $\alpha$ whose end points are $p$ and $q$. 

Let $\Gamma$ be a chart, and $m$ a label of $\Gamma$.
Let $\alpha$ be an arc in an edge of $\Gamma_m$, 
and 
$w$ a white vertex with $w\not\in\alpha$. 
Suppose that there exists an arc $\beta$ in $\Gamma$ such that
its end points are the white vertex $w$ 
and an interior point $p$ of the arc $\alpha$.
Then we say that 
{\it the white vertex $w$ connects with the point $p$ of $\alpha$ by the arc $\beta$}.

The following lemma will be used in Case (c) of the proof of Lemma~\ref{CorThreeRed}.

\begin{lemma}
{\em $($\cite[Lemma 4.2]{ChartAppl}$)$}
$($Shifting Lemma$)$
\label{Shift} 
Let $\Gamma$ be a chart and
$\alpha$ an arc in an edge of $\Gamma_m$.
Let
$w$ be a white vertex of $\Gamma_k \cap\Gamma_{h}$ where $h=k+\varepsilon ,\varepsilon\in\{+1,-1\}$.
Suppose that 
the white vertex $w$ connects with a point $p$ 
of the arc $\alpha$ by an arc in an edge $e$ of $\Gamma_k$. 
Suppose that
one of the following two conditions is satisfied: 
\begin{enumerate}
\item[$(1)$] $h>k>m$.
\item[$(2)$] $h<k<m$.
\end{enumerate}
Then for any neighborhood $V$ of the arc $e[w,p]$
we can shift the white vertex $w$
to $e-e[w,p]$ 
along the edge $e$
by C-I-R2 moves,
C-I-R3 moves and C-I-R4 moves in $V$ 
keeping $\displaystyle 
\bigcup_{i<0}\Gamma_{k+i\varepsilon}
$ 
fixed $($see Fig.~\ref{fig09}$)$.
\end{lemma}

\begin{figure}[t]
\centerline{\includegraphics{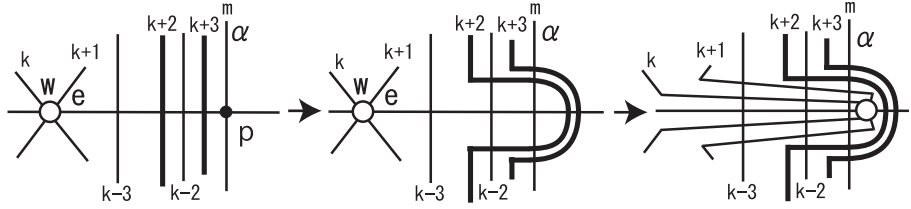}}
\vspace{5mm}
\caption{\label{fig09}Case (1): $k>m$ and $\varepsilon=+1$.}
\end{figure}

The following lemma will be used in
Case (d) in the proof of Lemma~\ref{CorThreeRed}
and Step~$3$, Step~$6$ and Step~$8$ in the proof of Theorem~\ref{NoLoop}.

\begin{lemma}
{\rm $($\cite[Lemma 5.4]{ChartAppl}$)$} 
\label{CorDiskLemma}
If a minimal chart $\Gamma$ contains the pseudo chart as shown in Fig.~\ref{fig10}, 
then the interior of the disk $D^*$ contains at least one white vertex, where $D^*$ is the disk with the boundary $e_3^* \cup e_4^* \cup e^*$.
\end{lemma}

\begin{figure}[t]
\centerline{\includegraphics{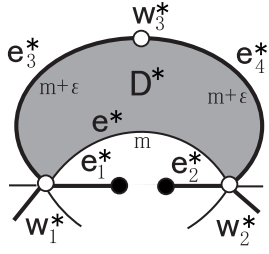}}
\vspace{5mm}
\caption{\label{fig10}The gray region is the disk $D^*$. The label of the edge $e^*$ is $m$,
and $\varepsilon\in\{+1,-1\}$.}
\end{figure}

We use the following notation:

In our argument,
we often need a name for an unnamed edge by using a given edge and a given white vertex.
For the convenience,
we use the following naming:
Let $e',e_i,e''$ be three consecutive edges containing  a white vertex $w_j$. Here, 
the two edges $e'$ and $e''$ are unnamed edges. 
There are six arcs in a neighborhood $U$ of the white vertex $w_j$. 
If the three arcs $e'\cap U$, $e_i \cap U$, $e'' \cap U$ lie anticlockwisely around the white vertex $w_j$ in this order, 
then $e'$ and $e''$ are denoted by $a_{ij}$ and $b_{ij}$ 
respectively (see Fig.~\ref{fig11}).
There is a possibility $a_{ij}=b_{ij}$ if they are contained in a loop.

\begin{figure}[t]
\centerline{\includegraphics{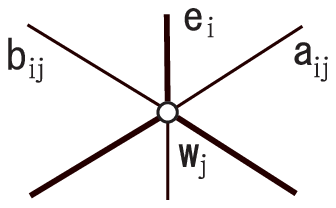}}
\vspace{5mm}
\caption{\label{fig11}}
\end{figure}

Let $\Gamma$ be a chart,
 and $v$ a vertex. 
Let $\alpha$ be a short arc of $\Gamma$ in a small neighborhood of $v$ with $v\in \partial \alpha$. 
If the arc $\alpha$ is oriented to $v$, then $\alpha$ is called {\it an inward arc}, 
and otherwise $\alpha$ is called {\it an outward arc}. 
 
Let $\Gamma$ be an $n$-chart. 
Let $F$ be a closed domain with $\partial F\subset \Gamma_{k-1}\cup\Gamma_{k}\cup \Gamma_{k+1}$ for some label $k$ of $\Gamma$, where $\Gamma_0=\emptyset$ and $\Gamma_{n}=\emptyset$. 
By Condition (iii) for charts,
in a small neighborhood of each white vertex, there are three inward arcs and three outward arcs.
Also in a small neighborhood of each black vertex, there exists only one inward arc or one outward arc.
We often use the following fact, 
when we fix (inward or outward) arcs 
near white vertices and black vertices: 
\begin{enumerate}
\item[]
{\it The number of inward arcs contained in $F\cap \Gamma_k$ is equal to the number of outward arcs in $F\cap \Gamma_k$.
}
\end{enumerate}
When we use this fact, 
we say that we use {\it IO-Calculation with respect to $\Gamma_k$ in $F$}.
For example, in a minimal chart $\Gamma$, 
consider the pseudo chart as shown in Fig.~\ref{fig12} where $\mu\in\{+1,-1\}$ and
\begin{enumerate}
\item[(1)] $D$ is the associated disk of a loop $\ell$ of label $k$ with $w(\Gamma\cap$Int$D)=3$,
\item[(2)]  $w_2,w_3,w_4$ are white vertices in Int$D$  with  $w_2,w_3,w_4\in\Gamma_{k+\mu}$,
\item[(3)] none of the four edges $a_{33},b_{33},a_{44},b_{44}$ contains a middle arc at $w_3$ nor $w_4$ (by Assumption~\ref{NoTerminal}
none of them is a terminal edge). 
\end{enumerate}
By (2), we have
$w_i\in\Gamma_{k}$ or $w_i\in\Gamma_{k+2\mu}$
for each $i=2,3,4$.
Let $F=Cl(D-D_1)$.
Then we can show that
\begin{enumerate}
\item[] if $w_2\in\Gamma_{k+2\mu}$,
then $w_3,w_4\in\Gamma_{k}$ or
$w_3,w_4\in\Gamma_{k+2\mu}$.
\end{enumerate}
For if not,
then one of $w_3,w_4$ is in $\Gamma_k$
and the other is in $\Gamma_{k+2\mu}$.
If $w_3\in\Gamma_k$ and $w_4\in\Gamma_{k+2\mu}$,
then 
considering 
$\partial D=\ell$ contains one inward arc and one outward arc,
by (3)
the number of inward arcs in $F\cap \Gamma_{k}$ is one,  
but the number of outward arcs in $F\cap \Gamma_k$ is three. 
This is a contradiction. 
Similarly if $w_3\in\Gamma_{k+2\mu}$ and $w_4\in\Gamma_{k}$,
then we have the same contradiction.
Thus if $w_2\in\Gamma_{k+2\mu}$,
then $w_3,w_4\in\Gamma_{k}$ or $w_3,w_4\in\Gamma_{k+2\mu}$.
Instead of the above argument, 
we just say that 
\begin{enumerate}
\item[]
{\it if $w_2\in\Gamma_{k+2\mu}$, 
then we have $w_3,w_4\in\Gamma_{k}$ or $w_3,w_4\in\Gamma_{k+2\mu}$ 
by IO-Calculation with respect to $\Gamma_{k}$ in $F$.}
\end{enumerate}

\begin{figure}[t]
\centerline{\includegraphics{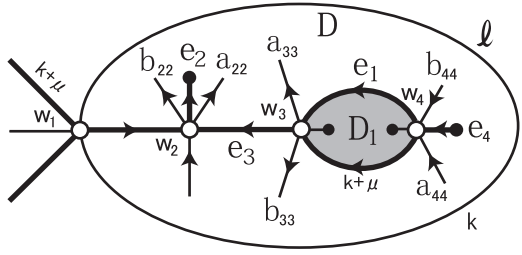}}
\vspace{5mm}
\caption{\label{fig12} The thick lines are edges of label $k+\mu$
where $\mu\in\{+1,-1\}$. 
The gray region is a 2-angled disk $D_1$.}
\end{figure}

 In our argument  we often construct a chart $\Gamma$. 
On the construction of a chart $\Gamma$, for a white vertex $w$,  
among the three edges of $\Gamma_m$ containing $w$, 
if one of the three edges is a terminal edge (see Fig.~\ref{fig13}(a) and (b)), 
then we remove the terminal edge and
put a black dot at the center of the white vertex  as shown in Fig.~\ref{fig13}(c).
Namely
Fig.~\ref{fig13}(c) means Fig.~\ref{fig13}(a) or Fig.~\ref{fig13}(b).

\begin{figure}[t]
\centerline{\includegraphics{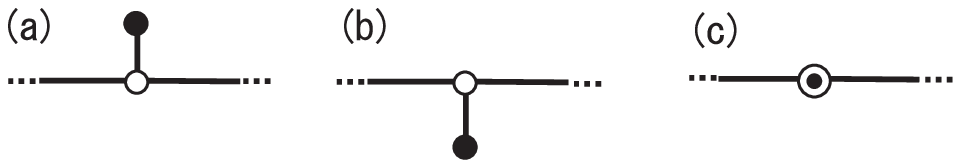}}
\vspace{5mm}
\caption{\label{fig13}}
\end{figure}


\section{{\large Proof of Lemma~3.2}}
\label{s:ProofLemma}

We start to prove Lemma~\ref{CorThreeRed}.
Let $\Gamma$ be a minimal chart 
with a loop $\ell$ of label $k$ 
such that $\Gamma$ is locally minimal with respect to the associated disk $D$ of $\ell$.
Suppose $w(\Gamma\cap$Int$D)\le 3$.
By Lemma~\ref{LoopTwoVertices}, 
we only need to examine 
the case $w(\Gamma\cap$Int$D)=3$.
In this section we assume that
\begin{enumerate}
\item[(1)] $w_1$ is the white vertex 
in the loop $\ell$ of label $k$, and
\item[(2)] $\mu$ is the integer in $\{+1,-1\}$ 
with $w_1\in\Gamma_{k+\mu}$.
\end{enumerate}
By Lemma~\ref{LoopTwoVertices}(1),
the condition $w(\Gamma\cap$Int$D)=3$ implies that 
the disk $D$ contains at most one loop of label $k+\mu$. 
Thus we can easily prove that 
the disk $D$ contains 
a pseudo chart of the RO-families of 
the six pseudo charts 
as shown in Fig.~\ref{fig14}. 
We use the notations as shown 
in Fig.~\ref{fig14} 
throughout this section.
There are six cases (see Fig.~\ref{fig14}).\\

\begin{figure}[t]
\centerline{\includegraphics{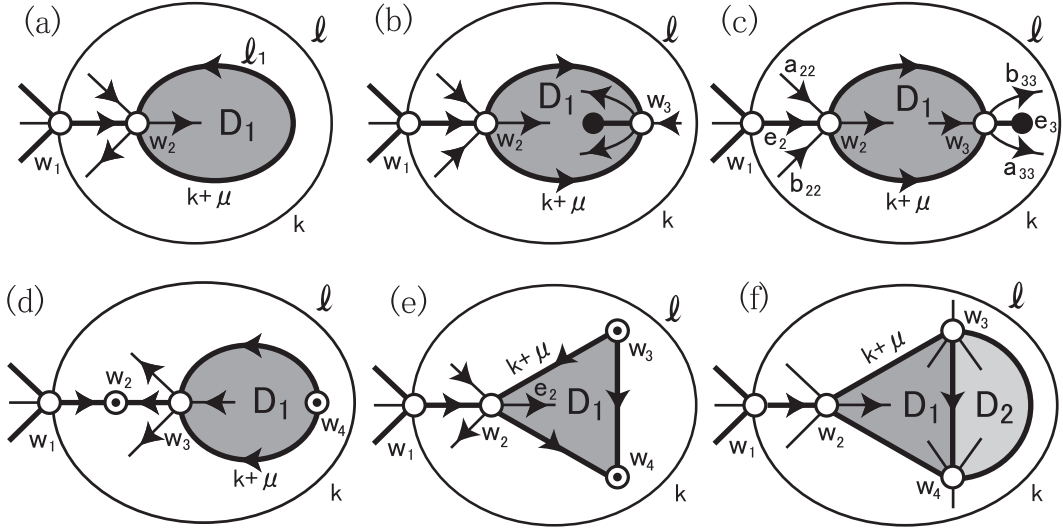}}
\vspace{5mm}
\caption{\label{fig14}
The thick lines are edges of label $k+\mu$
where $\mu\in\{+1,-1\}$.}
\end{figure}

{\bf Case (a).} 
Suppose that the disk $D$ contains 
the pseudo chart as shown in 
Fig.~\ref{fig14}(a),
where
\begin{enumerate}
\item[(a-1)]
the loop $\ell_1$ is 
of label $k+\mu$, and 
bounds a disk $D_1$.
\end{enumerate}
Since $w(\Gamma\cap$Int$D)=3$,
by Lemma~\ref{LoopTwoVertices}(1)
the disk $D_1$ contains exactly 
two white vertices in its interior,
say $w_3,w_4$.
Further by
 Lemma~\ref{LoopTwoVertices}(2)
the disk $D_1$ contains 
a pseudo chart of the RO-family of 
the pseudo chart 
as shown in Fig.~\ref{fig05}(a). 
Comparing Fig.~\ref{fig05}(a) and Fig.~\ref{fig14}(a), 
we have $D^*=D_1,m=k+\mu,w=w_2$ and
\begin{enumerate}
\item[(a-2)]
$e_1^*,e_2^*\subset
\Gamma_{k+\mu+\varepsilon}$,
$e_3^*, e_4^*, e_5^*, e_6^*\subset \Gamma_{k+\mu+2\varepsilon}$ where $\varepsilon\in\{+1,-1\}$
and
\item[(a-3)]
none of $e_3^*, e_4^*, e_5^*, e_6^*$ 
contains a middle arc at 
$w_3$ nor $w_4$
(hence none of $e_3^*, e_4^*, e_5^*, e_6^*$ is a terminal edge 
by Assumption~\ref{NoTerminal}).
\end{enumerate}
We claim that all of $e_3^*, e_4^*, e_5^*, e_6^*$ intersect $\ell=\partial D$.
If not,
then there exists a lens in $D$ not containing any white vertices in its interior,
because $e_3^*\not=e_6^*$ and $e_4^*\not=e_5^*$.
This contradicts Lemma~\ref{PolandLemma}.
Hence all of the four edges must 
intersect $\partial D$.
Thus $2\le |(k+\mu+2\varepsilon)-k|=
|\mu+2\varepsilon|$.
Hence $\varepsilon=\mu$
because $\varepsilon,\delta\in\{+1,-1\}$.
Thus by (a-2), 
we have 
$e_1^*, e_2^*\subset\Gamma_{k+2\mu}$, and
$e_3^*, e_4^*, e_5^*, e_6^*\subset \Gamma_{k+3\mu}$.
Therefore we have the pseudo chart as shown in Fig.~\ref{fig06}(a).\\


{\bf Case (b).}  Suppose that the disk $D$ 
contains the pseudo chart 
as shown in Fig.~\ref{fig14}(b), where
\begin{enumerate}
\item[(b-1)]
$D_1$ is a 2-angled disk of $\Gamma_{k+\mu}$ with {\it one feeler}.
\end{enumerate}
Since $w(\Gamma\cap$Int$D)=3$,
we have $w(\Gamma\cap$Int$D_1)\le1$.
Thus by Lemma~\ref{Theorem2AngledDisk},
the disk $D_1$ contains a pseudo chart of the RO-family of the pseudo chart 
as shown in Fig.~\ref{fig07}(c)
because the only pseudo chart contains a feeler.
Comparing Fig.~\ref{fig07}(c) and Fig.~\ref{fig14}(b), 
we have $m=k+\mu,w=w_2$ and
\begin{enumerate}
\item[(b-2)]
$e_1^*, e_2^*\subset\Gamma_{k+\mu+\varepsilon}$ where $\varepsilon\in\{+1,-1\}$,
\item[(b-3)]
$e_3^*, e_4^*\subset\Gamma_{k+\mu+2\delta}$,
$e_5^*\subset\Gamma_{k+\mu}$, $e_6^*,e_7^*\subset\Gamma_{k+\mu+\delta}$
where $\delta\in\{+1,-1\}$.
\end{enumerate}

We shall show $\mu=\varepsilon=\delta$.
If $\varepsilon=-\mu$, 
then we have
$e_1^*, e_2^*\subset \Gamma_k$.
Thus $e_1^*,e_2^*\subset Cl(D-D_1)$.
Hence whether $e_7^*\subset \Gamma_k$ or not,
we have a contradiction 
by 
IO-Calculation with respect to $\Gamma_{k}$ in $Cl(D-D_1)$.
Thus $\varepsilon=\mu$.
If $\delta=-\mu$, then we have
$e_3^*, e_4^*\subset\Gamma_{k-\mu}$.
Thus $e_3^*,e_4^*\subset D$.
Since $e_1^*, e_2^*\subset\Gamma_{k+2\mu}$ and
since $e_7^*\subset\Gamma_k$, 
we have a contradiction 
by IO-Calculation with respect to 
$\Gamma_{k-\mu}$ in $D$.
Thus we have $\delta=\mu$.
Hence by (b-2) and (b-3), we have
\begin{enumerate}
\item[(b-4)]
$e_1^*, e_2^*\subset\Gamma_{k+2\mu}$,
$e_3^*, e_4^*\subset\Gamma_{k+3\mu}$,
and
$e_6^*, e_7^*\subset\Gamma_{k+2\mu}$.
\end{enumerate}
Therefore we have the pseudo chart as shown in Fig.~\ref{fig06}(b).\\


{\bf Case (c).} 
Suppose that the disk $D$ contains the
pseudo chart as shown in Fig.~\ref{fig14}(c) 
where
\begin{enumerate}
\item[(c-1)] $D_1$ is a 2-angled disk of $\Gamma_{k+\mu}$ without feelers,
\item[(c-2)] the edges in $\partial D_1$ are oriented from $w_2$ to $w_3$, and
\item[(c-3)] 
each of edges $a_{22},b_{22},a_{33},b_{33}$ is of label $k$ or ${k+2\mu}$.
\end{enumerate}

Let $w_4$ be the white vertex in Int$D$ different from $w_2$ and $w_3$.
We show that
we can push $w_4$ out $D$ by C-moves.
Since $w(\Gamma\cap$Int$D)=3$,
we have $w(\Gamma\cap$Int$D_1)\le1$.
Thus 
considering the orientations of edges in $\partial D_1$,
by (c-1), (c-2) 
and Lemma~\ref{Theorem2AngledDisk},
the disk $D_1$ contains a pseudo chart of the RO-families of the two pseudo charts as shown in 
Fig.~\ref{fig07}(a) and (b).
Hence $w_4\in D-D_1$.
Let $s$ be the label with $w_4\in\Gamma_s\cap\Gamma_{s+1}$.
There are three cases:
\begin{enumerate}
\item[(i)]
for each $j=s,s+1$
 none of edges of $\Gamma_j$
containing $w_4$ intersects $\partial D$,
\item[(ii)] for each $j=s,s+1$
 there exists an edge of $\Gamma_j$
containing $w_4$ and intersecting $\partial D$,\item[(iii)] otherwise.
\end{enumerate}

{\bf Case (i).}
By Lemma~\ref{ComponentTwoWhite},
for each $j=s,s+1$
the connected component of $\Gamma_j$ containing $w_4$ contains another white vertex $w_2$ or $w_3$.
Hence there exists an edge $e_j'$ of $\Gamma_j$ with $e_j'\ni w_4$ containing $w_2$ or $w_3$. 
Here one of $e_s'$ and $e_{s+1}'$ is $a_{22}$ or $b_{22}$
and the other is $a_{33}$ or $b_{33}$.
Thus 
by (c-3),
the difference of labels of $e_s'$  and $e_{s+1}'$ is $0$ or $\pm2$.
This is a contradiction,
because $e_s'$ is of label $s$ and $e_{s+1}'$ is of lable ${s+1}$.
Hence Case (i) does not occur.

{\bf Case (ii).}
We have  $|s-k|\ge2$ and $|(s+1)-k|\ge2$,
because $\ell=\partial D$ is of label $k$.
Hence we can show $s<s+1<k$ or $k<s<s+1$.
By Shifting Lemma(Lemma~\ref{Shift})
we can push the white vertex $w_4$ 
from Int$D$ 
to the exterior of $D$ by C-moves.
However this contradicts the fact that 
$\Gamma$ is locally minimal 
with respect to $D$. 
Hence Case (ii) does not occur.

{\bf Case (iii).}
Let $e$ be the edge $\Gamma_s$ or $\Gamma_{s+1}$ with $e\ni w_4$ and $e\cap \partial D\not=\emptyset$, 
and $t$ the label of $e$.
Then $t\in\{s,s+1\}$.
Let $j\in \{s,s+1\}$ be the label different from $t$.
Then $t=j+\delta$ for some $\delta\in\{+1,-1\}$
and none of edges of $\Gamma_j$ containing $w_4$ intersects $\partial D$.
In a similar way to Case (i),
we can show that 
there exists an edge $e_j'$ of $\Gamma_j$
with $e_j'\ni w_4$
containing $w_2$ or $w_3$.
Here $e_j'$ is $a_{22},b_{22},a_{33}$ or $b_{33}$.
By (c-3),
we have $j=k$ or 
$j=k+2\mu$.
Thus we have $t=k+\delta$ or $t=k+2\mu+\delta$.
 Since the edge $e$ of label $t$ intersects the loop $\ell=\partial D$ of label $k$,
we have $|t-k|\ge2$.
Hence $\mu=\delta$ and $t=k+3\mu$.
Thus 
\begin{enumerate}
\item[(c-4)]
$w_4\in\Gamma_{k+2\mu}\cap\Gamma_{k+3\mu}$.
\end{enumerate}
There exists an arc $\beta$ in a regular neighborhood of 
$e_j'\cup D_1\cup e_2\cup e_3$ in $D$
connecting the vertex $w_4$ and a point $x$ in $\partial D$ such that 
Int$\beta$ intersects $\Gamma$ transversely
and  
\begin{enumerate}
\item[(c-5)] $\beta\cap\Gamma_{k+\mu}=\emptyset$.
\end{enumerate}
Let $p$ be a point in Int$\beta$ with $\Gamma\cap\beta[w_4,p]=w_4$.
By C-I-R2 moves,
in a regular neighborhood $N$ of $\beta[p,x]$
we can push all the arcs of $N\cap(\cup_{i\ge2}\Gamma_{k+i\mu})$ from $D$ along $\beta$
so that
 $(\beta-w_4)\cap\Gamma_{k+i\mu}=\emptyset$ 
($i\ge 2$),
and by (c-4) and (c-5)
we can push the white vertex $w_4$ along
$\beta$ 
from Int$D$ to the exterior of $D$ 
by C-I-R2 moves and C-I-R4 moves
(cf. \cite[Corollary 4.5]{ChartAppl}).
However again 
this contradicts the fact 
$\Gamma$ is locally minimal 
with respect to $D$.
Hence the disk $D$ does not contain 
the pseudo chart as shown 
in Fig.~\ref{fig14}(c).\\


{\bf Case (d).} 
Suppose that 
the disk $D$ contains the pseudo chart as shown in Fig.~\ref{fig14}(d).
Since $w(\Gamma\cap$Int$D_1)=0$,
considering the orientation of edges on $\partial D_1$,
 Lemma~\ref{Theorem2AngledDisk}
assures that 
the disk $D_1$ 
contains a pseudo chart of the RO-families of the two pseudo charts as shown in Fig.~\ref{fig07}(a) and (b).
So we have the pseudo chart as shown in Fig.~\ref{fig12} used for the explanation of IO-Calculation, here
\begin{enumerate}
\item[(d-1)] none of edges $a_{22},b_{22},a_{33},b_{33},a_{44},b_{44}$ contains a middle arc at $w_2,w_3$ nor $w_4$ (by Assumption~\ref{NoTerminal} none of them is a terminal edge).
\end{enumerate}

Since $w_2\in\Gamma_{k+\mu}$, 
we have also $w_2\in\Gamma_k$ or 
$w_2\in\Gamma_{k+2\mu}$.
If $w_2\in\Gamma_k$,
then $w_4\in\Gamma_k$
by IO-Calculation with respect to $\Gamma_k$ in $Cl(D-D_1)$. 
Hence $w_3\in\Gamma_{k+2\mu}$
by IO-Calculation with respect to $\Gamma_k$ in $Cl(D-D_1)$.
Thus by (d-1), 
we have $a_{22}=b_{44}$ and $b_{22}=a_{44}$.
Since $w_2,w_3,w_4$ are all the white vertices in Int$D$,
the disk $D_2$ bounded by $a_{22}\cup e_1\cup e_3$ does not contain any white vertices in its interior.
This contradicts Lemma~\ref{CorDiskLemma}
by setting $D^*=D_2, e_1^*=e_4, e_2^*=e_2, e^*=a_{22}, e_3^*=e_1, e_4^*=e_3$
in Fig.~\ref{fig10}. 
Hence $w_2\in\Gamma_{k+2\mu}$.
Thus 
as we have done in the explanation of IO-Calculation,
 we have $w_3,w_4\in\Gamma_k$ or 
 $w_3,w_4\in\Gamma_{k+2\mu}$ 
by IO-Calculation with respect to $\Gamma_k$ in $Cl(D-D_1)$.
If $w_3,w_4\in\Gamma_k$, 
then there exist two lenses of type $(k,k+\mu)$ in $D$ not containing any white vertices in their interiors.
This contradicts Lemma~\ref{PolandLemma}.
Thus $w_3,w_4\in\Gamma_{k+2\mu}$.
Looking at Fig.~\ref{fig12},
 we have the pseudo chart as shown in Fig.~\ref{fig06}(c). \\


{\bf Case (e).} 
Suppose that 
the disk $D$ contains the pseudo chart as shown in Fig.~\ref{fig14}(e) where
\begin{enumerate}
\item[(e-1)] $w_3,w_4$ are contained in terminal edges of label $k+\mu$, say $e_3,e_4$,
\item[(e-2)] the edge $e_2$ does not contain a middle arc at $w_2$ (hence $e_2$ is not a terminal edge by Assumption~\ref{NoTerminal}),
\item[(e-3)] the edge $e_2$ contains an outward arc at $w_2$.
\end{enumerate}
By (e-1), 
any feeler of $D_1$ is a terminal edge.
Thus the disk $D_1$ is a special 3-angled disk of $\Gamma_{k+\mu}$.
Hence $w(\Gamma\cap$Int$D_1)=0$ and Lemma~\ref{Theorem3AngledDisk}
imply that 
the disk $D_1$ contains a pseudo chart of the RO-families of the two pseudo charts as shown in Fig.~\ref{fig08}
each of which has no feeler.
Hence neither $e_3$ nor $e_4$ is a feeler.
That is
\begin{enumerate}
\item[(e-4)]  $e_3\cup e_4\subset Cl(D-D_1)$.
\end{enumerate}
Comparing Fig.~\ref{fig08} and Fig.~\ref{fig14}(e), we have $m=k+\mu$, and 
by (e-2) 
the edge $e_2$ is an edge of label $m+\varepsilon$ with two white vertices
where $\varepsilon\in\{+1,-1\}$. 
Hence
\begin{enumerate}
\item[(e-5)] the edge $e_2$ is an edge of $\Gamma_{k+\mu+\varepsilon}$ with 
$\partial e_2=\{w_2,w_3\}$ or  
$\partial e_2=\{w_2,w_4\}$.
\end{enumerate}
If $\partial e_2=\{w_2,w_3\}$,
then by (e-3)
the edge $e_2$ is oriented from $w_2$ to $w_3$.
This contradicts Condition (iii) of the definition of a chart.
Hence $\partial e_2=\{w_2,w_4\}$ (see Fig.~\ref{CaseE}).
Thus
\begin{enumerate}
\item[(e-6)] $w_2,w_4\in\Gamma_{k+\mu+\varepsilon}$.
\end{enumerate}
On the other hand,
by (e-1),
we have
\begin{enumerate}
\item[(e-7)] $w_3\in\Gamma_{k+\mu}\cap\Gamma_k$ or $w_3\in\Gamma_{k+\mu}\cap\Gamma_{k+2\mu}$.
\end{enumerate}
Suppose $\mu=-\varepsilon$.
Then $w_2,w_4\in\Gamma_k$ by (e-6).
If $w_3\in\Gamma_k$, then 
we have $a_{33}=b_{44}$ in Fig.~\ref{CaseE}
because the loop $\ell$ is of label $k$.
Hence there exists a lens of type $(k,k+\mu)$ in $Cl(D-D_1)$ containing $w_3,w_4$
but not containing any white vertices in its interior.
This contradicts Lemma~\ref{PolandLemma}.
Hence $w_3\in\Gamma_{k+2\mu}$  by (e-7).
But we have a contradiction by IO-Calculation with respect to $\Gamma_k$ in $Cl(D-D_1)$.
Thus $\mu=\varepsilon$.

By (e-6),
we have $w_2,w_4\in\Gamma_{k+2\mu}$.
If $w_3\in\Gamma_k$,
then we have a contradiction by IO-Calculation with respect to $\Gamma_k$ in $Cl(D-D_1)$.
Hence $w_3\in\Gamma_{k+2\mu}$  by (e-7). 
Therefore we have the pseudo chart as shown 
in Fig.~\ref{fig06}(d).\\

\begin{figure}[t]
\centerline{\includegraphics{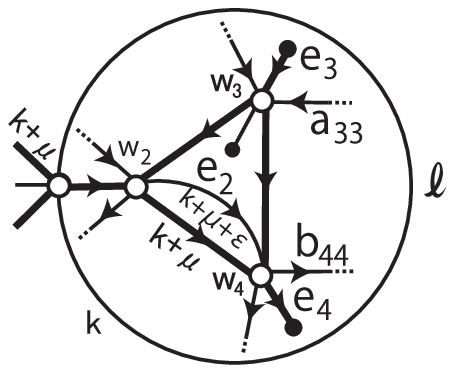}}
\vspace{5mm}
\caption{\label{CaseE} The thick lines are edges of label $k+\mu$ where $\mu,\varepsilon\in\{+1,-1\}$.}
\end{figure}


{\bf Case (f).}
Suppose that the disk $D$ contains the pseudo chart as shown in Fig.~\ref{fig14}(f).
Since $w(\Gamma\cap$Int$D_2)=0$, 
by Lemma~\ref{Theorem2AngledDisk}
the 2-angled disk $D_2$ contains a pseudo chart of the RO-families of the two pseudo charts as shown in Fig.~\ref{fig07}(a) and (b).
Hence both the two edges in $\partial D_2$ are oriented from $w_3$ to $w_4$.
Thus 
\begin{enumerate}
\item[(f-1)] the boundary of the 3-angled disk $D_1$ is oriented clockwise.
\end{enumerate}

On the other hand,
 the disk $D_1$ does not have any feelers. 
Thus the disk $D_1$ is special.
Since $w(\Gamma\cap$Int$D_1)=0$, 
by Lemma~\ref{Theorem3AngledDisk}
the disk $D_1$ contains a pseudo chart of the RO-families of the two pseudo charts  as shown in Fig.~\ref{fig08}. This contradicts Condition (f-1).
Hence $D$ does not contain the pseudo chart as shown in Fig.~\ref{fig14}(f). 

Therefore we complete the proof of Lemma~\ref{CorThreeRed}. \hfill $\square$


\section{{\large Solar eclipses}}
\label{s:SolarEclipse}

\begin{lemma}
\label{NoLens}
{\em (\cite[Corollary 1.3]{ChartAppII})}
Let $\Gamma$ be a minimal chart with at most seven white vertices.
Then there is no lens of $\Gamma$.
\end{lemma}

A subgraph of a chart is called 
a {\it solar eclipse} if 
it consists of two loops and contains
only one white vertex.

\begin{lemma}
\label{ConditionSolarEclipse}
{\em (\cite[Lemma 8.4]{ChartAppII})}
Let $\Gamma$ be a minimal chart with at most seven white vertices. 
If there exists a solar eclipse, 
then the associated disk of each loop of the solar eclipse contains at least three white vertices in its interior. 
\end{lemma}

\begin{lemma}
\label{Table}
Let $D$ be the associated disk of a loop $\ell$ of label $m$ in a minimal chart $\Gamma$ with $w(\Gamma)=7$
such that $\Gamma$ is locally minimal with respect to $D$. 
Let $\varepsilon$ be the integer in $\{+1,-1\}$
such that the white vertex in $\ell$ is in $\Gamma_{m+\varepsilon}$.
Then we have the following:
\begin{enumerate}
\item[$(1)$]
If $w(\Gamma\cap${\rm Int}$D)=2$,
then $w(\Gamma_{m+2\varepsilon}\cap(S^2-D))\ge2$.
\item[$(2)$]
If $w(\Gamma\cap${\rm Int}$D)=3$,
then $w(\Gamma_{m+2\varepsilon}\cap(S^2-D))\ge1$.
\item[$(3)$]
If $w(\Gamma\cap${\rm Int}$D)=3$
and if the disk $D$ contains one of the two pseudo charts as shown in Fig.~\ref{fig06}{\rm (a)} and {\rm (b)} here $m=k$ and $\varepsilon=\mu$,
then $w(\Gamma_{m+3\varepsilon}\cap(S^2-D))\ge1$.
\end{enumerate}
\end{lemma}

\begin{proof}
We show Statement (1).
By Lemma~\ref{LoopTwoVertices}(2),
the disk $D$ contains a pseudo chart of the RO-family of the pseudo chart as shown in Fig.~\ref{fig05}(a)
here $D=D^*$.
Since none of the four edges $e_3^*,e_4^*,e_5^*,e_6^*$ of $\Gamma_{m+2\varepsilon}$ in Fig.~\ref{fig05}(a) contains a middle arc at a white vertex in Int$D$,
all of four edges $e_3^*,e_4^*,e_5^*,e_6^*$ are not terminal edges
by Assumption~\ref{NoTerminal}.

If $w(\Gamma_{m+2\varepsilon}\cap (S^2-D))=0$,
then $e_3^*=e_5^*$ and $e_4^*=e_6^*$.
Thus there exist two lenses of type $(m+\varepsilon,m+2\varepsilon)$. 
This contradicts Lemma~\ref{NoLens}. 

Now suppose that $w(\Gamma_{m+2\varepsilon}\cap(S^2-D))=1$.
Let $D_1$ be the 2-angled disk of $\Gamma_{m+2\varepsilon}$ in $D$.
Since $w(\Gamma_{m+2\varepsilon}\cap(S^2-D))=1$
and $w(\Gamma_{m+2\varepsilon}\cap(D-D_1))=0$,
there exists only one white vertex of $\Gamma_{m+2\varepsilon}$ in $S^2-D_1$.
Thus we have a contradiction by IO-Calculation with respect to $\Gamma_{m+2\varepsilon}$
in $Cl(S^2-D_1)$.
Hence $w(\Gamma_{m+2\varepsilon}\cap(S^2-D))\ge2$. 

We show Statement (3).
Suppose that the disk $D$ contains the pseudo chart as
shown in Fig.~\ref{fig06}(a)
here $k=m$ and $\mu=\varepsilon$.
Then the loop $\ell_1$ of label $m+\varepsilon$ in $D$ bounds the associated disk $D_1$ with $w(\Gamma\cap$Int$D_1)=2$.
By Lemma~\ref{Table}(1),
we have $w(\Gamma_{(m+\varepsilon)+2\varepsilon}\cap(S^2-D_1))\ge2$.
Since $w(\Gamma_{m+3\varepsilon}\cap(D-D_1))=0$, 
we have $w(\Gamma_{m+3\varepsilon}\cap(S^2-D))\ge2$.
Thus clearly
 $w(\Gamma_{m+3\varepsilon}\cap(S^2-D))\ge1$.

Suppose that the disk $D$ contains the pseudo chart as
shown in Fig.~\ref{fig06}(b)
here $k=m$ and $\mu=\varepsilon$.
By Lemma~\ref{ComponentTwoWhite},
the condition $w(\Gamma_{m+3\varepsilon}\cap D)=1$ implies $w(\Gamma_{m+3\varepsilon}\cap(S^2-D))\ge1$.

We show Statement (2).
Now suppose that
\begin{enumerate}
\item[(i)] $w(\Gamma_{m+2\varepsilon}\cap(S^2-D))=0$.
\end{enumerate}
Since $w(\Gamma\cap$Int$D)=3$,
 by Lemma~\ref{CorThreeRed}
the disk $D$ contains a pseudo chart of the RO-families of the four pseudo charts as
shown in Fig.~\ref{fig06} where $k=m$ and $\mu=\varepsilon$. 

Suppose that the disk $D$ contains the pseudo chart as
shown in Fig.~\ref{fig06}(a).
By (i),
we have a loop $\ell_2$ of label $m+2\varepsilon$ so that $\ell_1\cup \ell_2$ is a solar eclipse.
But the associated disk of $\ell_1$ contains only two white vertices in its interior.
This contradicts Lemma~\ref{ConditionSolarEclipse}.

Suppose that the disk $D$ contains the pseudo chart as
shown in Fig.~\ref{fig06}(b).
Let $D_1$ be the $2$-angled disk of $\Gamma_{m+\varepsilon}$ in $D$. 
Since $w(\Gamma_{m+2\varepsilon}\cap (D-D_1))=0$, the condition (i) implies that 
$w(\Gamma_{m+2\varepsilon}\cap(S^2-D_1))=0$.
Hence we have a contradiction by IO-Calculation with respect to $\Gamma_{m+2\varepsilon}$
in $Cl(S^2-D_1)$.

Suppose that the disk $D$ contains the pseudo chart as
shown in Fig.~\ref{fig06}(c).
Let $D_1$ be the $2$-angled disk of $\Gamma_{m+\varepsilon}$ in $D$. 
Since there exists only one white vertex  
of $\Gamma_{m+2\varepsilon}$ in $D-D_1$,
the condition (i) implies that
 $w(\Gamma_{m+2\varepsilon}\cap(S^2-D_1))=1$.
Hence we have a contradiction
by IO-Calculation with respect to $\Gamma_{m+2\varepsilon}$ in $Cl(S^2-D_1)$.

Suppose the disk $D$ contains the pseudo chart as
shown in Fig.~\ref{fig06}(d).
Let $D_1$ be the 3-angled disk of $\Gamma_{m+\varepsilon}$ in $D$.
Since $w(\Gamma_{m+2\varepsilon}\cap (D-D_1))=0$, the condition (i) implies that  
$w(\Gamma_{m+2\varepsilon}\cap(S^2-D_1))=0$.
Hence
for the edge $e_1'$ in Fig.~\ref{fig06}(d)
we have $e_1'=e_2'$, $e_1'=e_3'$ or $e_1'=e_4'$.
If $e_1'=e_2'$,
then there exists a lens of type $(m+\varepsilon,m+2\varepsilon)$ containing $w_3,w_4$.
This contradicts Lemma~\ref{NoLens}.
If $e_1'=e_3'$ or $e_1'=e_4'$,
then the edge $e_1'$ separates the disk $Cl(S^2-D_1)$ into two disks.
Let $D'$ be the disk of the two disks
containing the edge $e_2'$.
Then we have a contradiction by IO-Calculation with respect to $\Gamma_{m+2\varepsilon}$ in $D'$.

Therefore we have a contradiction for any cases. Hence $w(\Gamma_{m+2\varepsilon}\cap(S^2-D))\ge1$. 
\end{proof}


\begin{proposition}
\label{NoSolarEclipse}
There is no solar eclipse in a minimal chart with exactly seven white vertices.
\end{proposition}

\begin{proof}
Suppose that there exists a solar eclipse $G$ in a minimal chart $\Gamma$ with $w(\Gamma)=7$.
Then $G$ consists of two loops $\ell_1$ and $\ell_2$ 
with $\ell_1\subset \Gamma_m$ and $\ell_2\subset\Gamma_{m+1}$ for some label $m$. 
We can assume that
$\Gamma$ is locally minimal with respect to the associated disks $D_1$ and $D_2$ of the loops $\ell_1$ and $\ell_2$ respectively.
By Lemma~\ref{ConditionSolarEclipse},
we have $w(\Gamma\cap$Int$D_1)\ge3$ and $w(\Gamma\cap$Int$D_2)\ge3$. 
Since $w(\Gamma)=7$ and $w(\ell_1\cup \ell_2)=1$,
we have $w(\Gamma\cap$Int$D_1)=3$ and $w(\Gamma\cap$Int$D_2)=3$.
By Lemma~\ref{CorThreeRed},
the two disks $D_1$ and $D_2$ contain pseudo charts of the RO-families of the four pseudo charts as shown in Fig.~\ref{fig06}. 
For $\ell_2\subset \Gamma_{m+1}$
where $\ell=\ell_2,k=m+1,\mu=-1$ in Fig.~\ref{fig06} we have
\begin{enumerate}
\item[(1)]  $w((\Gamma_{m}\cup \Gamma_{m-1})\cap$Int$D_2)=3$. 
\end{enumerate}
Since $S^2-D_1\supset$Int$D_2$, by (1)
we have
$w((\Gamma_m\cup\Gamma_{m-1})\cap(S^2-D_1))\ge w((\Gamma_m\cup\Gamma_{m-1})\cap$Int$D_2)=3$. 
Setting $\ell=\ell_1\subset\Gamma_m$
and $\varepsilon=1$
(because the white vertex in $\ell_1$ is in $\Gamma_{m+1}$),
applying Lemma~\ref{Table}(2)
we have $w(\Gamma_{m+2}\cap(S^2-D_1))\ge1$.
 Hence we have\\
$\begin{array}{rl}
w(\Gamma\cap(S^2-D_1)) &\ge 
w((\Gamma_m\cup\Gamma_{m-1})\cap(S^2-D_1))+w(\Gamma_{m+2}\cap(S^2-D_1))\\
&\ge 3+1=4.
\end{array}$\\
Thus
\begin{center}
$w(\Gamma)= w(\ell_1)+w(\Gamma\cap$Int$D_1)+w(\Gamma\cap(S^2-D_1))\ge 1+3+4=8$.
\end{center}
This contradicts the fact  $w(\Gamma)=7$.
Therefore there is no solar eclipse in $\Gamma$.
\end{proof}

\begin{lemma}
\label{ThreeVertexExt}
Let $\ell$ be a loop of label $m$ in a minimal chart $\Gamma$ with $w(\Gamma)=7$,
and $D$ the associated disk of $\ell$.
Then $w(\Gamma_m\cap(S^2-D))\le2$.
\end{lemma}

\begin{proof}
Suppose that $w(\Gamma_m\cap(S^2-D))\ge3$.
We shall show $w(\Gamma)\ge8$. 
It suffices to prove the case that $\Gamma$ is locally minimal with respect to $D$.
Let $\varepsilon$ be the integer in $\{+1,-1\}$ such that 
the white vertex in $\ell$ is in $\Gamma_{m+\varepsilon}$.
By Lemma~\ref{LoopTwoVertices}(1)
we have $w(\Gamma\cap$Int$D)\ge2$.

If $w(\Gamma\cap$Int$D)=2$,
then by Lemma~\ref{Table}(1) 
we have $w(\Gamma_{m+2\varepsilon}\cap(S^2-D))\ge2$.
Since $w(\Gamma_m\cap(S^2-D))\ge3$,
we have $w(\Gamma\cap(S^2-D))\ge3+2$. Thus
\begin{center}
$w(\Gamma)= w(\ell)+w(\Gamma\cap$Int$D)+w(\Gamma\cap(S^2-D))\ge 1+2+(3+2)=8$.
\end{center}
If $w(\Gamma\cap$Int$D)=3$,
then by Lemma~\ref{Table}(2) 
we have $w(\Gamma_{m+2\varepsilon}\cap(S^2-D))\ge1$.
Similarly we can show $w(\Gamma)\ge 1+3+(3+1)=8$.
If $w(\Gamma\cap$Int$D)\ge4$,
then we can show $w(\Gamma)\ge1+4+(3+0)=8$.
Hence we have $w(\Gamma)\ge8$.
This contradicts the fact $w(\Gamma)=7$.
Therefore $w(\Gamma_m\cap(S^2-D))\le2$.
\end{proof}


\section{{\large Pairs of eyeglasses and pairs of skew eyeglasses}}
\label{s:Eyeglasses}

\begin{lemma}
\label{TwoThreeLoop}
{\em (cf. \cite[Lemma 6.2]{ChartAppII})}
Let $\Gamma$ be a minimal chart,
and $m$ a label of $\Gamma$.
Let $G$ be a connected component of $\Gamma_m$
containing a white vertex and a loop. 
If $G$ contains at most three white vertices,
then $G$ is one of the three subgraphs as shown 
in Fig.~\ref{fig15}. 
\end{lemma}

\begin{figure}[t]
\centerline{\includegraphics{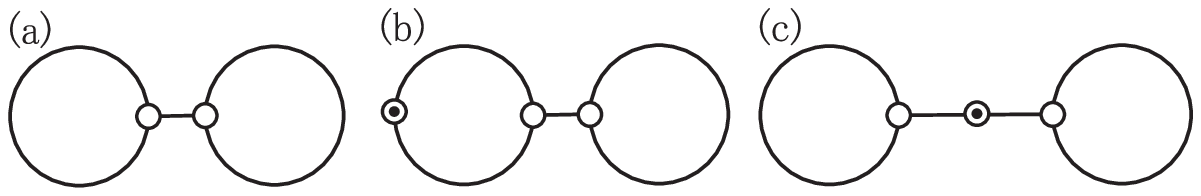}}
\vspace{5mm}
\caption{\label{fig15}}
\end{figure}

We call the subgraph as shown in Fig.~\ref{fig15}(a)
{\it a pair of eyeglasses}.
We call the subgraphs as shown in Fig.~\ref{fig15}(b) and (c)
 {\it pairs of skew eyeglasses of type $1$ and $2$} respectively.

\begin{lemma}
\label{LemmaNoGlasses}
Let $G$ be a pair of eyeglasses in a minimal chart $\Gamma$ with $w(\Gamma)=7$.
Let $D_1,D_2$ be the associated disks of loops in $G$.
If $\Gamma$ is locally minimal with respect to the disks $D_1$ and $D_2$,
then 
$w(\Gamma\cap${\rm Int}$D_1)=2$ and $w(\Gamma\cap${\rm Int}$D_2)=2$.
\end{lemma}

\begin{proof}
Since $w(\Gamma\cap${\rm Int}$D_1)+w(\Gamma\cap${\rm Int}$D_2)+w(G)\le w(\Gamma)=7$,
we have $w(\Gamma\cap${\rm Int}$D_1)+w(\Gamma\cap${\rm Int}$D_2)\le 5$.
Thus by Lemma~\ref{LoopTwoVertices}(1)
we have 
\begin{center}
$(w(\Gamma\cap${\rm Int}$D_1),w(\Gamma\cap${\rm Int}$D_2))=(2,2),(2,3)$ or $(3,2)$.
\end{center}

Suppose that $(w(\Gamma\cap${\rm Int}$D_1),w(\Gamma\cap${\rm Int}$D_2))=(2,3)$.
Then by $w(\Gamma)=7$, we have
\begin{enumerate}
\item[(1)]  $w(\Gamma\cap(S^2-(D_1\cup D_2)))=0$.
\end{enumerate}
Further
by Lemma~\ref{LoopTwoVertices}(2)
the disk $D_1$ contains a pseudo chart of the RO-family of the pseudo chart as shown in Fig.~\ref{fig05}(a) (see Fig.~\ref{fig16}(a)). 
Without loss of generality
we can assume that the edge $e_1$ of $\Gamma_m$ is oriented from $w_2$ to $w_1$.
We use the notations as shown in Fig.~\ref{fig16}(a).
Hence 
\begin{enumerate}
\item[(2)] $w(\Gamma_{m+\varepsilon}\cap D_1)=3$.
\end{enumerate}

One of the two edges $a_{11},b_{11}$ of $\Gamma_{m+\varepsilon}$ in Fig.~\ref{fig16}(a) does not contain a middle arc at $w_1$,
one of the two edges is not a terminal edge
by Assumption~\ref{NoTerminal}.
Thus by (1),
there are three cases:
 $a_{11}=b_{11}, a_{11}\ni w_2$ or $b_{11}\ni w_2$.
But if $a_{11}=b_{11}$,
then $a_{11}\cup \ell_1$ is a solar eclipse.
This contradicts Proposition~\ref{NoSolarEclipse}.
Hence we have 
\begin{enumerate}
\item[(3)]
$w_2\in a_{11}\subset \Gamma_{m+\varepsilon}$
or $w_2\in b_{11}\subset \Gamma_{m+\varepsilon}$.
\end{enumerate}

Since $w(\Gamma\cap$Int$D_2)=3$,
by Lemma~\ref{CorThreeRed}
the disk $D_2$ contains a pseudo chart of the RO-families of the four pseudo charts as shown in Fig.~\ref{fig06}. 
By (3), for the label $k$ and $\mu\in\{+1,-1\}$ in Fig.~\ref{fig06}
we have $k=m$ and $\mu=\varepsilon$. 

Suppose that $D_2$ contains one of the two pseudo charts as shown in Fig.~\ref{fig06}(a) and (b).
Then we have $w(\Gamma_{m+3\varepsilon}\cap(S^2-D_2))\ge1$ by Lemma~\ref{Table}(3).
Further $S^2-D_2\supset D_1$ and (2) imply that
$w((\Gamma_{m+\varepsilon}\cup\Gamma_{m+3\varepsilon})\cap(S^2-D_2))\ge3+1=4$.
Thus we have $w(\Gamma\cap(S^2-D_2))\ge4$.
Hence 
\begin{center}
$w(\Gamma)= w(\ell_2)+w(\Gamma\cap$Int$D_2)+w(\Gamma\cap(S^2-D_2))\ge 1+3+4=8$.
\end{center}
This contradicts the fact $w(\Gamma)=7$.

Suppose that $D_2$ contains the pseudo chart as shown in Fig.~\ref{fig06}(c)
(see Fig.~\ref{fig16}(b)).
Now let $E_i$ be the 2-angled disk of $\Gamma_{m+\varepsilon}$ in $D_i$ for $i=1,2$. 
Since $(D_1\cup D_2)-(E_1\cup E_2)$ contains only one white vertex of $\Gamma_{m+2\varepsilon}$,
the condition (1) implies
$w(\Gamma_{m+2\varepsilon}\cap(S^2-(E_1\cup E_2)))=1$.
Hence we have a contradiction
 by IO-Calculation with respect to $\Gamma_{m+2\varepsilon}$ in $Cl(S^2-(E_1\cup E_2))$.

Suppose that $D_2$ contains the pseudo chart as shown in Fig.~\ref{fig06}(d) (see Fig.~\ref{fig16}(c)) where
\begin{enumerate}
\item[(d-1)]  none of $e_1^*,e_2^*,e_3^*$ in
Fig.~\ref{fig16}(c) contains a middle arc at $w_4^*$ or $w_5^*$
(by Assumption~\ref{NoTerminal},
none of $e_1^*,e_2^*,e_3^*$ is a terminal edge),
\item[(d-2)]
by (3)
all seven white vertices are contained in the same connected component of $\Gamma_{m+\varepsilon}$.
\end{enumerate}

{\bf Claim.}
All of $e_1^*,e_2^*,e_3^*$ contain white vertices in Int$D_1$.

For, by (1) and (d-1),
 for the edge $e_1^*$,
there are four cases:
 $e_1^*=e_4^*$, 
 $e_1^*=e_5^*$,
 the edge $e_1^*$ is a loop or
 the edge $e_1^*$ contains a white vertex in Int$D_1$.

If  $e_1^*=e_4^*$ or $e_1^*=e_5^*$,
the edge $e_1^*$ separates the disk $Cl(S^2-D_3)$ into two disks,
here $D_3$ is the $3$-angled disk of $\Gamma_{m+\varepsilon}$ in $D_2$ (see Fig.~\ref{fig16}(c)).
Let $D'$ be one of the two disks contains the edge $e_2^*$.
Then we have a contradiction by IO-Calculation with respect to  $\Gamma_{m+2\varepsilon}$ in $D'$.
If $e_1^*$ is a loop,
then the associated disk $D^*$ of $e_1^*$ contains $w_2$.
Further by (d-2),
the disk $D^*$ must contain the three white vertices in $D_1$.
Hence $S^2-(D_3\cup D^*)$ does not contain any white vertices.
Thus $e_3^*=e_4^*$ and $e_2^*=e_5^*$.
Hence there exists a lens of type $(m+\varepsilon,m+2\varepsilon)$.
This contradicts Lemma~\ref{NoLens}.
Therefore $e_1^*$ contains a white vertex in Int$D_1$.
Similarly we can show that 
each of the two edges $e_2^*,e_3^*$ contains a white vertex in Int$D_1$.
Thus Claim holds.

But 
 the three edges $e_1^*,e_2^*,e_3^*$ of $\Gamma_{m+2\varepsilon}$ are oriented outward at
the same white vertex in Int$D_1$.
This is a contradiction. 

Thus we have a contradiction for any cases.
Hence $(w(\Gamma\cap$Int$D_1),w(\Gamma\cap$Int$D_2))\not=(2,3)$.
In a similar way as above,
we can show that $(w(\Gamma\cap$Int$D_1),w(\Gamma\cap$Int$D_2))\not=(3,2)$.
Therefore $(w(\Gamma\cap$Int$D_1),w(\Gamma\cap$Int$D_2))=(2,2)$.
This completes the proof of Lemma~\ref{LemmaNoGlasses}.
\end{proof}

\begin{figure}[t]
\centerline{\includegraphics{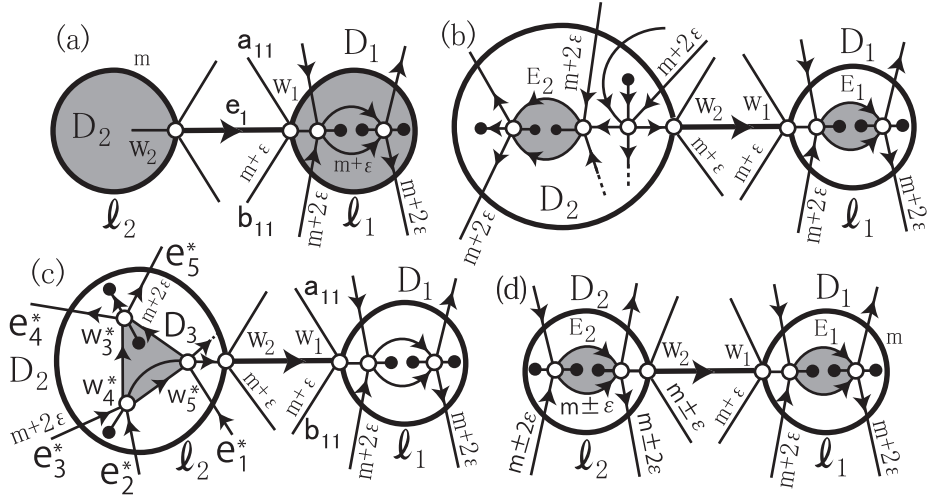}}
\vspace{5mm}
\caption{\label{fig16} (a) The gray regions are disks $D_1,D_2$.  (b), (d) The gray regions are disks $E_1,E_2$.
(c) The gray region is a $3$-angled disk $D_3$.}
\end{figure}


\begin{proposition}
\label{NoGlasses}
There is no pair of eyeglasses in a minimal chart with exactly seven white vertices.
\end{proposition}

\begin{proof}
Suppose that there exists a pair of eyeglasses $G$ in a minimal chart $\Gamma$ with $w(\Gamma)=7$. 
Without loss of generality
we can assume that $\Gamma$ is locally minimal with respect to the associated disks $D_1$ and $D_2$ of loops in $G$.
By Lemma~\ref{LemmaNoGlasses},
we have 
 $(w(\Gamma\cap$Int$D_1),w(\Gamma\cap$Int$D_2))=(2,2)$.
By Lemma~\ref{LoopTwoVertices}(2)
each of the disks $D_1$ and $D_2$ contains a pseudo chart of the RO-family of the pseudo chart as shown in Fig.~\ref{fig05}(a) (see Fig.~\ref{fig16}(d)).
Moreover there exists exactly one white vertex in $S^2-(D_1\cup D_2)$,
say $w'$.
 
If $w'\in\Gamma_{m}$,
then this contradicts Lemma~\ref{ComponentTwoWhite}. 
Thus $w'\not\in\Gamma_m$.

If $w'\in\Gamma_{m+2\varepsilon}$,
then there exists only one white vertex $w'$
of $\Gamma_{m+2\varepsilon}$ in $S^2-(E_1\cup E_2)$
where $E_i$ is the 2-angled disk $E_i$ in $D_i$ for $i=1,2$. 
Hence if $w_2\in\Gamma_{m+\varepsilon}$ (resp. $w_2\not\in\Gamma_{m+\varepsilon}$, i.e. $w_2\in\Gamma_{m-\varepsilon}$)
then we have a contradiction by IO-Calculation with respect to $\Gamma_{m+2\varepsilon}$ in
 $Cl(S^2-(E_1\cup E_2))$ (resp. $Cl(S^2-E_1)$).
Thus $w'\not\in\Gamma_{m+2\varepsilon}$.

Similarly we can show $w'\not\in\Gamma_{m-2\varepsilon}$.
Now $w'\not\in\Gamma_m\cup\Gamma_{m\pm2\varepsilon}$ implies $w'\not\in\Gamma_{m\pm\varepsilon}$,
because 
a white vertex in $\Gamma_k$
is also in $\Gamma_{k-1}$ or $\Gamma_{k+1}$. 
Hence the vertex $w'$ is contained in $\Gamma_k$ for some label $k$ with $k\not=m,m\pm\varepsilon,m\pm2\varepsilon$.
However $\Gamma_k$ contains exactly one white vertex $w'$.
This contradicts Lemma~\ref{ComponentTwoWhite}. 
Therefore there is no pair of eyeglasses in $\Gamma$.
\end{proof}


\begin{proposition}
\label{NoSkewGlassesType2}
There is no pair of skew eyeglasses of type $2$ in a minimal chart with exactly seven white vertices.
\end{proposition}

\begin{proof}
Suppose that there exists a pair of skew eyeglasses $G$ of type $2$ in a minimal chart $\Gamma$ with $w(\Gamma)=7$.  
We only show that the edge $e_3$ of $\Gamma_m$ is oriented inward at $w_3$.
We use the notations as shown in 
Fig.~\ref{fig17}(a).
By Lemma~\ref{LoopTwoVertices}(1), 
the condition $w(\Gamma)=7$
implies  $(w(\Gamma\cap$Int$D_1),w(\Gamma\cap$Int$D_2))=(2,2)$.
Thus
\begin{enumerate} 
\item[(i)] $w(\Gamma\cap(S^2-(G\cup D_1\cup D_2)))=0$.
\end{enumerate}
For the edge $a_{33}$ of $\Gamma_{m+\varepsilon}$ in Fig.~\ref{fig17}(a),
there are five cases:
(1) $a_{33}=a_{11}$, 
(2) $a_{33}=b_{11}$,
(3) $a_{33}=a_{22}$,
(4) $a_{33}=b_{22}$,
(5) $a_{33}$ is a loop.

{\bf Case (1).}
The union $a_{33}\cup e_1$ bounds a lens.
This contradicts Lemma~\ref{NoLens}.

{\bf Case (3).}
We have $b_{33}=b_{22}$.
Thus $b_{33}\cup e_2$ bounds a lens.
This contradicts Lemma~\ref{NoLens}.

{\bf Case (4).}
The union $a_{33}\cup e_2$ bounds a disk $D'$ containing the edge $b_{33}$.
By (i),
we have $w(\Gamma\cap$Int$D')=0$.
Thus we have a contradiction 
by IO-Calculation with respect to $\Gamma_{m+\varepsilon}$ in $D'$.

{\bf Case (5).}
We have $b_{33}\ni w_2$
because the edge $b_{33}$ does not contain a middle arc at $w_3$
(the edge $b_{33}$ is not a terminal edge 
by Assumption~\ref{NoTerminal}).
Thus $w_2\in \Gamma_{m+\varepsilon}$.
Hence whether $w_1\in\Gamma_{m-\varepsilon}$ or $w_1\in\Gamma_{m+\varepsilon}$,
we have $a_{11}=b_{11}$,
because one of $a_{11}$ and $b_{11}$ does not contain a middle arc.
Thus $a_{11}\cup \ell_1$ is a solar eclipse.
This contradicts Proposition~\ref{NoSolarEclipse}.

Hence Case (2) occurs.
Similarly we can show $b_{33}=a_{22}$.
Thus the two edges $a_{11},b_{22}$ are terminal edges.
By Lemma~\ref{LoopTwoVertices}(2),
both of $D_1$ and $D_2$ 
contain pseudo charts of the RO-family of the pseudo chart as shown in Fig.~\ref{fig05}(a) 
(see Fig.~\ref{fig17}(b)).
There are four edges of $\Gamma_{m+2\varepsilon}$ intersecting $\partial D_1$, say $e_1',e_2',e_3',e_4'$.
By Lemma~\ref{NoLens}, 
we have $e_i'\not=e_j'$ for $i\not=j$.
Thus the four edges must contain white vertices in Int$D_2$.
However it is impossible that Int$(e_i')\cap $Int$(e_j')=\emptyset$ for each pair $i,j$ with $1\le i<j\le4$.
Therefore there is no pair of skew eyeglasses of type $2$ in $\Gamma$.
\end{proof}

\begin{figure}[t]
\centerline{\includegraphics{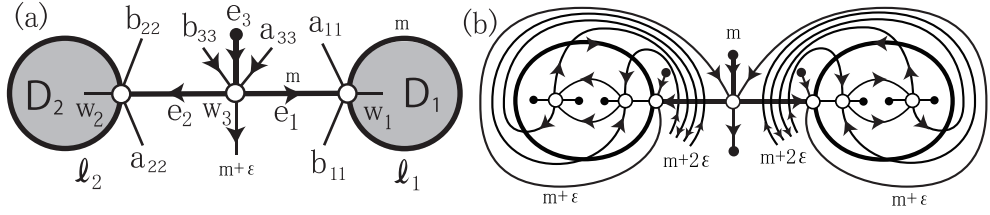}}
\vspace{5mm}
\caption{\label{fig17}
The gray regions are disks $D_1,D_2$, $m$ is a label, and $\varepsilon=\pm1$.}
\end{figure}

\section{{\large Triangle Lemma}}
\label{s:TriangleLemma}

Let $\Gamma$ be a chart, and $D$ a disk.
Let $\alpha$ be a simple arc in $\partial D$.
We call a simple arc $\gamma$ in an edge of $\Gamma_k$
a {\it {$(D,\alpha)$-arc}} of label $k$
provided that 
$\partial \gamma \subset $Int$\alpha$
and
Int$\gamma\subset $Int$D$. 
If there is no $(D,\alpha)$-arc in $\Gamma$,
then the chart $\Gamma$ is said to be
$(D,\alpha)$-{\it arc free}.

Let $\Gamma$ be a chart and 
$D$ a disk.
Let $\alpha$ be a simple arc in $\partial D$.
For each $k=1,2,\cdots$, 
let $\Sigma_k$ be the pseudo chart 
which consists of all arcs in $D\cap \Gamma_k$ 
intersecting the set $Cl(\partial D-\alpha)$.
Let $\Sigma_\alpha={\cup_k \Sigma_k}$.

The following two lemmata will be used in the proof of Lemma~\ref{LemmaTriangle}.

\begin{lemma}
{\em (\cite[Lemma 3.2]{ChartAppl})} 
$($Disk Lemma$)$
\label{DiskLemma}
Let $\Gamma$ be a minimal chart and
$D$ a disk.
Let $\alpha$ be a simple arc in $\partial D$.
Suppose that the interior of $\alpha$ contains neither white vertices, 
isolated points of $D\cap \Gamma$, nor arcs of $D \cap \Gamma$. 
If {\em Int}$D$ does not contain white vertices of $\Gamma$,
then for any neighborhood $V$ of $\alpha$, 
there exists a $(D,\alpha)$-arc free minimal chart $\Gamma'$ obtained from the chart $\Gamma$
by C-moves in $V\cup D$ keeping $\Sigma_\alpha$ fixed $($see Fig.~\ref{fig18}$)$.
\end{lemma}

\begin{figure}[t]
\centerline{\includegraphics{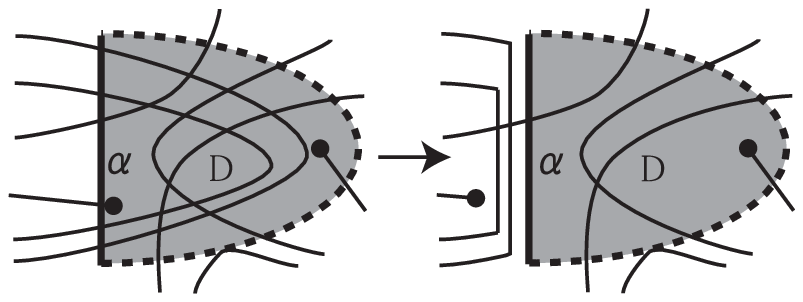}}
\vspace{5mm}
\caption{\label{fig18}}
\end{figure}

\begin{lemma}
\label{MM+2Edge}
Let $\Gamma$ be a chart,
$e$ an edge of $\Gamma_m$,
and $w_1,w_2$ the white vertices in $e$.
Suppose $w_1\in\Gamma_{m-1}$ and $w_2\in\Gamma_{m+1}$.
Then for any neighborhood $V$ of the edge $e$, there exists a chart $\Gamma'$
obtained from the chart $\Gamma$ by C-I-R2 moves, C-I-R3 moves and C-I-R4 moves
in $V$ keeping $\Gamma_{m-1}\cup\Gamma_m\cup\Gamma_{m+1}$
such that the edge $e$ does not contain any crossings.
\end{lemma}

\begin{proof}
 Let $b$ be a point in Int$(e)$.
Since Int$(e)\cap(\Gamma_{m-1}\cup\Gamma_{m+1})=\emptyset$,
all the crossings of $\Gamma_i$ ($i<m-1$) in the edge $e$ can be moved into the arc $e[b,w_2]$
by C-I-R2 moves and C-I-R3 moves.
Similarly all the crossings of $\Gamma_i$ ($i>m+1$) in the edge $e$ can be moved into 
the arc $e[w_1,b]$ by C-moves.

Since the white vertex $w_1$ is in $\Gamma_{m}\cap\Gamma_{m-1}$,
we can push each crossing in the arc $e[w_1,b]$ to the side of $w_1$ along the edge $e$ by C-I-R2 moves and C-I-R4 moves.
Similarly we can push each crossing in the arc $e[b,w_2]$ to the side of $w_2$ along the edge $e$ by C-moves.
Hence
 the edge $e$ does not contain any crossings.
\end{proof}

The following lemma will be used in Step~$1$ and Step~$8$ of the proof of 
Theorem~\ref{NoLoop}.


\begin{lemma}$($Triangle Lemma$)$
\label{LemmaTriangle}
\begin{enumerate}
\item[$(1)$]
For a chart $\Gamma$,
if there exists  a $3$-angled disk $D_1$ of $\Gamma_m$ without feelers in a disk $D$ as shown in Fig.~\ref{fig19}{\rm (a)} and if $w(\Gamma\cap${\rm Int}$D_1)=0$,
then there exists a chart obtained from $\Gamma$ by C-moves in $D$ which contains the pseudo chart in $D$ as shown in Fig.~\ref{fig19}{\rm (b)}. 
\item[$(2)$] For a minimal chart $\Gamma$, 
if there exists a $3$-angled disk $D_1$ of $\Gamma_m$ without feelers in a disk $D$ as shown in Fig.~\ref{fig19}{\rm (c)},
then $w(\Gamma\cap${\rm Int}$D_1)\ge1$.
\end{enumerate}
\end{lemma}

\begin{figure}[t]
\centerline{\includegraphics{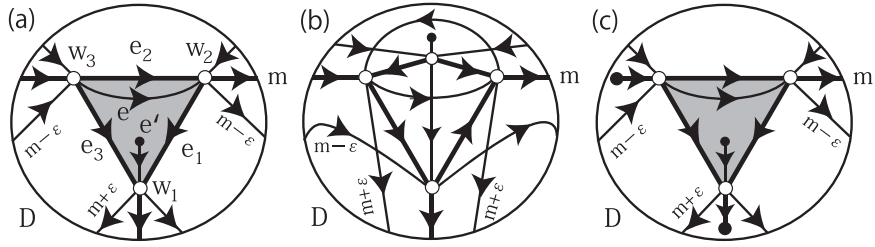}}
\vspace{5mm}
\caption{\label{fig19}The gray regions are 3-angled disks $D_1$. 
The thick lines are edges of label $m$,
and $\varepsilon\in\{+1,-1\}$.}
\end{figure}

\begin{proof}
We show Statement (1). 
We use the notations as shown in Fig.~\ref{fig19}(a).

{\bf Claim 1.} $D_1\cap(\Gamma_{m-\varepsilon}\cup\Gamma_{m}\cup\Gamma_{m+\varepsilon})= e\cup e'\cup\partial D_1$. 

For, if $D_1$ contains a ring or a hoop of label $m-\varepsilon,m,m+\varepsilon$, say $C$,
then $C$ bounds a disk in $D_1$ not containing any white vertices. 
This contradicts Assumption~\ref{AssumptionRing}.
Hence the disk $D_1$ contains neither hoop nor ring of label $m-\varepsilon,m,m+\varepsilon$. 
Since 
the disk $D_1$ does not contain any free edges by Assumption~\ref{NoSimpleHoop}, 
 Claim 1 holds.

{\bf Claim 2.} We can assume that neither $e_1$ nor $e$  contains a crossing by C-moves.

For,
$w_1\in\Gamma_{m+\varepsilon}$
and $w_2\in\Gamma_{m-\varepsilon}$.
Thus by Lemma~\ref{MM+2Edge}
we can assume that 
the edge $e_1$ of $\Gamma_{m}$
does not contain any crossings.
Moreover
we can push the crossings in $e$ to the side of $w_3$ by C-moves as follows.
Since the label of the edge $e$ is $m-\varepsilon$,
we have Int$(e)\cap(\Gamma_{m-2\varepsilon}\cup\Gamma_{m})=\emptyset$.
Hence Claim 1 implies that
we have Int$(e)\cap(\Gamma_{m-2\varepsilon}\cup\Gamma_{m}\cup\Gamma_{m+\varepsilon})=\emptyset$.
Since the white vertex $w_3$ is in $\Gamma_{m-\varepsilon}\cap\Gamma_{m}$,
we can push each crossing in the edge $e$ to the side of $w_3$ along $e$ by C-I-R2 moves and C-I-R4 moves.
Hence Claim 2 holds.

{\bf Claim 3.}  
We can assume that $D_1\cap\Gamma= e\cup e'\cup\partial D_1$
by C-moves.

For, Claim 2 assures that
by applying Disk Lemma(Lemma~\ref{DiskLemma}) twice,
we can assume that $\Gamma$ is $(D_1,e_2)$-arc free and $(D_1,e_3)$-arc free.
Since $D_1$ does not contain free edges, hoops nor rings
by Assumption~\ref{NoSimpleHoop} and Assumption~\ref{AssumptionRing},
we have $D_1\cap\Gamma= e\cup e'\cup\partial D_1$.
Hence Claim 3 holds.

As shown in Fig.~\ref{fig20},
we can deform by C-moves in $D$ from $\Gamma$ to a chart containing the pseudo chart as shown in Fig.~\ref{fig19}(b).

We show Statement (2).
Suppose that a minimal chart $\Gamma$ contains the pseudo chart as shown in Fig.~\ref{fig19}(c).
If $w(\Gamma\cap$Int$D_1)=0$,
then as shown in Fig.~\ref{fig21},
by the help of Triangle Lemma(Lemma~\ref{LemmaTriangle}(1)),
we can reduce the number of white vertices of $\Gamma$ by C-moves.
This contradicts the fact that $\Gamma$ is a minimal chart.
Hence $w(\Gamma\cap$Int$D_1)\ge1$.
\end{proof}

\begin{figure}[t]
\centerline{\includegraphics{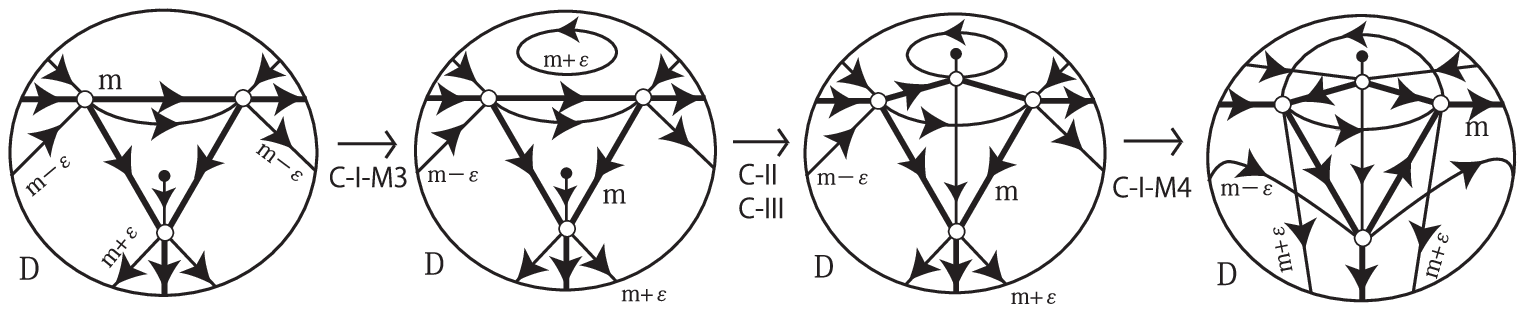}}
\vspace{5mm}
\caption{\label{fig20}The thick lines are edges of label $m$,
and $\varepsilon\in\{+1,-1\}$.}
\end{figure}

\begin{figure}[t]
\centerline{\includegraphics{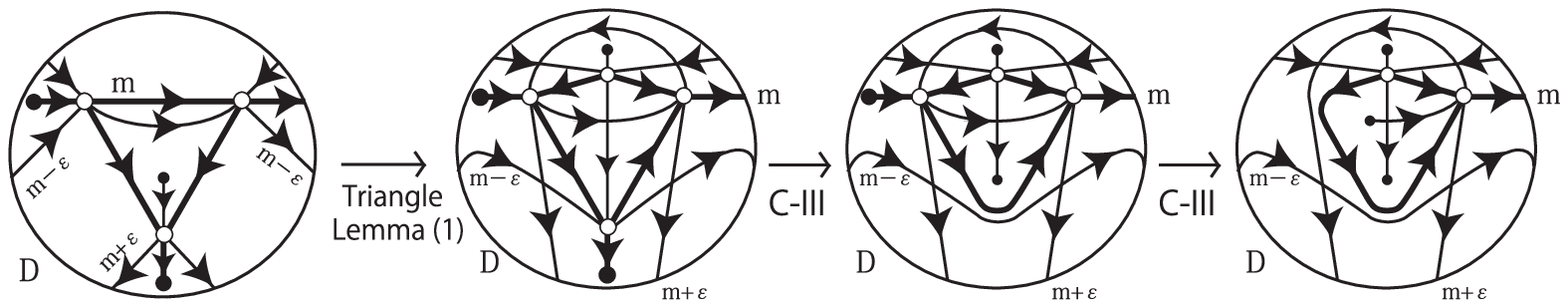}}
\vspace{5mm}
\caption{\label{fig21}The thick lines are edges of label $m$,
and $\varepsilon\in\{+1,-1\}$.}
\end{figure}

\section{{\large Proof of Theorem~1.1}}
\label{s:ProofTheorem}

Now we start to prove Theorem~\ref{NoLoop}.
We show the theorem by contradiction.
Suppose that there exists a minimal chart $\Gamma$ with $w(\Gamma)=7$
containing a loop $\ell$ of label $m$.
Let $D_1$ be the associated disk of $\ell$.
Then $w(\Gamma_m\cap(S^2-D_1))\le2$ by Lemma~\ref{ThreeVertexExt}.
Thus the loop $\ell$ is contained in a connected component $G$ of $\Gamma_m$ with $w(G)\le3$. By Lemma~\ref{TwoThreeLoop},
the graph $G$ is a pair of eyeglasses or a pair of skew eyeglasses.
By Proposition~\ref{NoGlasses} 
and Proposition~\ref{NoSkewGlassesType2},
the graph $G$ is a pair of skew eyeglasses of type $1$.

From now on,
throughout this section,
 we assume that
\begin{enumerate}
\item[(i)] $\ell$ is the loop of label $m$ in the pair of skew eyeglasses $G$ of type $1$, 
\item[(ii)] $w_1$ is the white vertex in  $\ell$ with $w_1\in\Gamma_m\cap\Gamma_{m+\varepsilon}$ where $\varepsilon\in\{+1,-1\}$, 
\item[(iii)] $\Gamma$ is locally minimal with respect to $D_1$ and $D_2$
where $D_1$ is the associated disk of $\ell$ and $D_2$ is the $2$-angled disk of $\Gamma_m$ with $w_1\not\in D_2$
(see Fig.~\ref{fig22}).
\end{enumerate}

\begin{figure}[t]
\centerline{\includegraphics{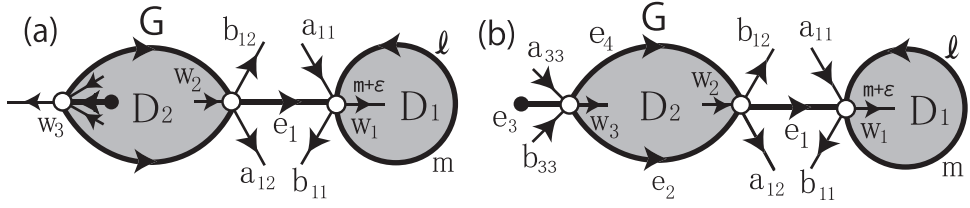}}
\vspace{5mm}
\caption{\label{fig22}
The pair of skew eyeglasses $G$ is the thick lines of label $m$,
 the gray regions are the disks $D_1,D_2$.}
\end{figure}

We shall prove Theorem~\ref{NoLoop} by 
eight steps:
If necessary
we change all orientation of edges,
we can assume
that the edge $e_1$ of label $m$ is oriented from $w_2$ to $w_1$. 
We show only the case that
$\ell$ is oriented clockwisely. 
We use the notations as shown in Fig.~\ref{fig22}.

{\bf Step 1.} The disk $D_2$ is a $2$-angled disk without feelers (i.e. Fig.~\ref{fig22}(b)).

{\bf Step 2.} There exists at least one white vertex of $\Gamma_{m-1}\cup\Gamma_{m+1}$ in $S^2-(D_1\cup D_2)$, say $w_7$.

{\bf Step 3.} 
$w(\Gamma\cap$Int$D_1)=2$.

{\bf Step 4.} There exists a white vertex in $S^2-(D_1\cup D_2)$ different from $w_7$, say $w_6$, and $w_6,w_7\in\Gamma_{m+\varepsilon}\cap\Gamma_{m+2\varepsilon}$ and
$w_2,w_3\in\Gamma_{m+\varepsilon}$.

{\bf Step 5.}
The edge $a_{12}$ contains one of the white vertices $w_6,w_7$.

Without loss of generality
we can assume $w_6\in a_{12}$.

{\bf Step 6.}
The edge $b_{33}$ contains the white vertex $w_7$.

{\bf Step 7.}
The edge $a_{33}$ contains the white vertex $w_7$.

{\bf Step 8.}
There does not exist a pair of skew eyeglasses of type $1$.\\

{\it Proof of Step~$1$.} Suppose that 
$D_2$ is a $2$-angled disk with one feeler.
We use the notations as shown in Fig.~\ref{fig22}(a).

We shall show that $\Gamma$ contains a pseudo chart as shown in Fig.~\ref{fig23}.
By IO-Calculation with respect to $\Gamma_{m\pm1}$ in $D_2$ and $Cl(S^2-(D_1\cup D_2))$,
there exists a white vertex of $\Gamma_{m-1}\cup\Gamma_{m+1}$ in Int$D_2$, say $w_6$, and 
there exists a white vertex of $\Gamma_{m-1}\cup\Gamma_{m+1}$ in $S^2-(D_1\cup D_2)$, say $w_7$.
Now 
\begin{enumerate}
\item[(1)]
$\begin{array}{rl}
7= & w(\Gamma)\\
 = & w(\Gamma\cap{\rm Int}D_1)+w(\Gamma\cap{\rm Int}D_2)+w(G)+w(\Gamma\cap(S^2-(D_1\cup D_2)))\\
\ge & w(\Gamma\cap{\rm Int}D_1)+1+3+1.
\end{array}$
\end{enumerate}
Hence we have $w(\Gamma\cap$Int$D_1)\le 2$.
By Lemma~\ref{LoopTwoVertices}(1),
we have $w(\Gamma\cap$Int$D_1)= 2$.
Thus by (1) 
\begin{enumerate}
\item[(2)]$w(\Gamma\cap$Int$D_2)=1$ and $w(\Gamma\cap(S^2-(D_1\cup D_2)))=1$.
\end{enumerate}
Since $D_2$ is a $2$-angled disk with one feeler, 
by (2) and Lemma~\ref{Theorem2AngledDisk}
 the disk $D_2$ contains a pseudo chart of the RO-family of the pseudo chart as shown in Fig.~\ref{fig07}(c) (see $D_2$ in Fig.~\ref{fig23}).

For the edge $b_{11}$ in Fig.~\ref{fig22}(a),
there are two cases:
$b_{11}=a_{11}$ or $b_{11}\ni w_7$.
But if $b_{11}=a_{11}$,
then 
we have a solar eclipse $b_{11}\cup\ell$.
This contradicts Proposition~\ref{NoSolarEclipse}.
Hence $b_{11}\ni w_7$.

For the edge $a_{12}$ in Fig.~\ref{fig22}(a),
there are two cases:
 $a_{12}=a_{11}$ or $a_{12}\ni w_7$.
If $a_{12}=a_{11}$,
then
the edges $e_1$ and $a_{12}$ separate the annulus $Cl(S^2-(D_1\cup D_2))$ into two regions.
Let $F$ be one of the two regions with $w_7\not\in F$.
By (2), we have $w(\Gamma\cap F)=0$.
Thus we have a contradiction by IO-Calculation with respect to $\Gamma_{m+\varepsilon}$ in $F$. Hence $a_{12}\ni w_7$.

Since $a_{12}$ and $b_{11}$ are oriented inward at $w_7$,
we have $b_{12}\not\ni w_7$.
Thus we have $b_{12}=a_{11}$,
because $b_{12}$ does not contain a middle arc at $w_2$.
Let $e_7$ be the edge of $\Gamma_{m+\varepsilon}$ containing $w_7$ different from $a_{12}$ and $b_{11}$. 
Then (2) implies that $e_7$ is a terminal edge.

We shall show $w_7\in\Gamma_{m+2\varepsilon}$.
Since $w_7\in b_{11}\subset\Gamma_{m+\varepsilon}$,
we have $w_7\in\Gamma_{m}$ or $w_7\in\Gamma_{m+2\varepsilon}$.
If $w_7\in\Gamma_{m}$,
then there exists a connected component of $\Gamma_m$ containing exactly one white vertex $w_7$.
This contradicts Lemma~\ref{ComponentTwoWhite}.
Thus $w_7\in\Gamma_{m+2\varepsilon}$. 

Let $D_3$ be the $3$-angled disk bounded by $a_{11}\cup b_{11}\cup a_{12}$
with $e_1\subset D_3$,
If $e_7\subset D_3$,
then we have a contradiction by IO-Calculation with respect to $\Gamma_{m+2\varepsilon}$ in $D_3$.
Thus we have $e_7\not\subset D_3$ 
(see Fig.~\ref{fig23}).
Since $w_7\in S^2-(D_1\cup D_2)$,
(2) implies that
\begin{enumerate}
\item[(3)]
$w(\Gamma\cap$Int$D_3)=0$.
\end{enumerate}
Thus the edge of $\Gamma_{m+2\varepsilon}$ containing $w_7$ in $D_3$ is a terminal edge.
Therefore 
$\Gamma$ contains a pseudo chart as shown in Fig.~\ref{fig23}.

However in a neighborhood of $D_3$,
the chart $\Gamma$ contains the pseudo chart as shown in Fig.~\ref{fig19}(c).
Now (3) contradicts Triangle Lemma(Lemma~\ref{LemmaTriangle}(2)).
Hence $D_2$ is a $2$-angled disk without feelers (see Fig.~\ref{fig22}(b)).
This completes the proof of Step~$1$.
\hfill$\square$\\

\begin{figure}[t]
\centerline{\includegraphics{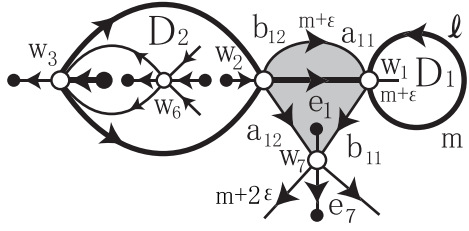}}
\vspace{5mm}
\caption{\label{fig23}The gray region is a $3$-angled disk $D_3$.}
\end{figure}


{\it Proof of Step~$2$.} 
By Step~$1$,
the chart $\Gamma$ contains the pseudo chart as shown in Fig.~\ref{fig22}(b).
We use the notations as shown 
in Fig.~\ref{fig22}(b).
Suppose that there does not exist any white vertices of $\Gamma_{m-1}\cup \Gamma_{m+1}$ 
in $S^2-(D_1\cup D_2)$.
By Assumption~\ref{NoTerminal},
the edge $b_{11}$ is not a terminal edge.
Thus
there are three cases:
(1) $b_{11}=a_{11}$, (2) $b_{11}=a_{33}$, (3) $b_{11}=b_{33}$.

{\bf Case (1).}
We have a solar eclipse $b_{11}\cup \ell$.
This contradicts Proposition~\ref{NoSolarEclipse}.

{\bf Case (2).}
We have $b_{33}=a_{12}$.
Hence the union $b_{33}\cup e_2$ bounds a lens.
This contradicts Lemma~\ref{NoLens}.

{\bf Case (3).} 
We have $w_1,w_3\in\Gamma_{m+\varepsilon}$, 
because $\partial b_{11}=\{w_1,w_3\}$.
By Assumption~\ref{NoTerminal},
the edge $a_{12}$ is not a terminal edge.
Hence $a_{12}=a_{11}$ or $a_{12}=a_{33}$.
Thus $a_{12}$ intersects the edge $b_{11}$.
Since $w_1\in a_{12}\cap b_{11}$ or $w_3\in a_{12}\cap b_{11}$,
the two edges $a_{12}$ and $b_{11}$
are of label $m+\varepsilon$.
This contradicts Condition (iv) of the definition of a chart.

Hence we have a contradiction for any cases.
Thus there exists a white vertex of $\Gamma_{m-1}\cup \Gamma_{m+1}$ 
in $S^2-(D_1\cup D_2)$.
This completes the proof of Step~$2$.
\hfill $\square$\\

{\it Proof of Step~$3$.}
By Step~$2$
\begin{enumerate}
\item[(*)]$\begin{array}{rl}
7=& w(\Gamma)\\
\ge& w(\Gamma\cap{\rm Int}D_1)+w(\Gamma\cap{\rm Int}D_1)+w(G)+w(\Gamma\cap(S^2-(D_1\cup D_2))\\
\ge& w(\Gamma\cap{\rm Int}D_1)+0+3+1.
\end{array}$
\end{enumerate}
Thus $w(\Gamma\cap$Int$D_1)\le 3$.
By Lemma~\ref{LoopTwoVertices}(1),
 we have $w(\Gamma\cap$Int$D_1)\ge2$.
Hence it suffices to prove $w(\Gamma\cap$Int$D_1)\not=3$.

Suppose $w(\Gamma\cap$Int$D_1)=3$.
By (*), Step~$2$ and $w(\Gamma)=7$,
we have
\begin{enumerate}
\item[(**)] $w(\Gamma\cap(S^2-(D_1\cup D_2))=1$
and $w(\Gamma\cap$Int$D_2)=0$.
\end{enumerate}
Let $w_7$ be the white vertex in $S^2-(D_1\cup D_2)$.
For the edge $b_{33}$ in Fig.~\ref{fig22}(b),
there are four cases:
(1) $b_{33}=a_{12}$,
(2) $b_{33}=b_{12}$,
(3) $b_{33}=b_{11}$, 
(4) $b_{33}\ni w_7$.

{\bf Case (1).}
The union $b_{33}\cup e_2$ bounds a lens. 
This contradicts Lemma~\ref{NoLens}.

{\bf Case (2).}
The edge $b_{33}$ separates the annulus $Cl(S^2-(D_1\cup D_2))$ into a disk and an annulus.
Let $F$ be one of the disk and the annulus
with $w_7\not\in F$.
Since $w(\Gamma\cap(S^2-(D_1\cup D_2))=1$ by (**),
we have $w(\Gamma\cap F)=0$.
Thus we have a contradiction by IO-Calculation with respect to $\Gamma_{m\pm1}$ in $F$.

{\bf Case (3).}
We have $w_1,w_3\in \Gamma_{m+\varepsilon}$.
For the edge $a_{12}$ in Fig.~\ref{fig22}(b),
the edge is of label $m-\varepsilon$ or $m+\varepsilon$.
If $a_{12}$ is of label ${m+\varepsilon}$,
then we have $a_{12}\ni w_7$ and $a_{33}=b_{12}$,
because neither $a_{12}$ nor $a_{33}$ is a terminal edge 
by Assumption~\ref{NoTerminal}.
Thus the union $a_{33}\cup e_4$ bounds a lens.
This contradicts Lemma~\ref{NoLens}.
Hence $a_{12}$ is of label ${m-\varepsilon}$.
Since $a_{12}$ is not a terminal edge,
 we have $w_7\in a_{12}\subset\Gamma_{m-\varepsilon}$.
Hence $w_7\in\Gamma_{m-2\varepsilon}$ or $w_7\in\Gamma_{m}$. If $w_7\in\Gamma_m$,
then there exists a connected component of $\Gamma_m$
containing exactly one white vertex $w_7$.
This contradicts Lemma~\ref{ComponentTwoWhite}.
Hence $w_7\in\Gamma_{m-2\varepsilon}$.

We shall show that
$\Gamma_{m-2\varepsilon}$ contains exactly one white vertex $w_7$.
Since $w(\Gamma\cap$Int$D_1)=3$, 
$w_1\in\Gamma_{m+\varepsilon}$,
and since $\Gamma$ is locally minimal with respect to $D_1$
by (iii),
all the white vertices in Int$D_1$
are contained in $\Gamma_{m+\varepsilon}\cup\Gamma_{m+2\varepsilon}$
by Lemma~\ref{CorThreeRed}.
Thus none of the white vertices in Int$D_1$
is contained in $\Gamma_{m-2\varepsilon}$.
Since $w_1,w_2,w_3\in\Gamma_m$,
none of $w_1,w_2,w_3$ 
is contained in $\Gamma_{m-2\varepsilon}$.
Hence $\Gamma_{m-2\varepsilon}$ contains exactly one white vertex $w_7$.
This contradicts Lemma~\ref{ComponentTwoWhite}.

{\bf Case (4).}
For the edge $b_{11}$ in Fig.~\ref{fig22}(b),
there are three cases: $b_{11}=a_{11}$, $b_{11}=a_{33}$ or $b_{11}\ni w_7$.
If $b_{11}=a_{11}$,
then we have a solar eclipse $b_{11}\cup \ell$.
This contradicts Proposition~\ref{NoSolarEclipse}.
If $b_{11}=a_{33}$ or $b_{11}\ni w_7$,
then we have $a_{12}\ni w_7$ (see Fig.~\ref{fig24}(a)).
Since $w(\Gamma\cap$Int$D_2)=0$ by (**),
the disk $D_2$
contains a pseudo chart of the RO-family of
the pseudo chart as shown in Fig.~\ref{fig07}(b)
by Lemma~\ref{Theorem2AngledDisk}.
Since $w(\Gamma\cap(S^2-(D_1\cup D_2))=1$ by (**),
we have $w(\Gamma\cap$Int$D_3)=0$, 
here $D_3$ is the disk with $\partial D_3=e_2\cup a_{12}\cup b_{33}$ and $w_1\not\in D_3$.
This contradicts Lemma~\ref{CorDiskLemma}.

Thus we have a contradiction for any cases.
Hence $w(\Gamma\cap$Int$D_1)\not=3$.
Therefore $w(\Gamma\cap$Int$D_1)=2$.
This completes the proof of Step~$3$.
\hfill $\square$\\

{\it Proof of Step~$4$.}
By Step~$3$,
we have $w(\Gamma\cap$Int$D_1)=2$.
Since $w(\Gamma)=7$,
we have 
\begin{enumerate}
\item[(1)] $w(\Gamma\cap(S^2-D_1))=4$.
\end{enumerate}
Since $w_2,w_3\in S^2-D_1$ and $w_7\in S^2-D_1$ by Step~$2$,
 there exists a white vertex in $S^2-D_1$ different from $w_2,w_3,w_7$, say $w_6$.
By Lemma~\ref{Table}(1)
we have $w(\Gamma_{m+2\varepsilon}\cap(S^2-D_1))\ge2$.
Since $w_2,w_3\in\Gamma_m$ (see Fig.~\ref{fig22}(b))
we have $w_2,w_3\not\in\Gamma_{m+2\varepsilon}$.
Hence we have 
\begin{enumerate}
\item[(2)]$w_6,w_7\in\Gamma_{m+2\varepsilon}$.
\end{enumerate}
Since $w_7\in\Gamma_{m-1}\cup\Gamma_{m+1}$ by Step~$2$,
we have $w_7\in\Gamma_{m+\varepsilon}\cap\Gamma_{m+2\varepsilon}$.

By Step~$2$ and (1),
we have 
\begin{center}
$\begin{array}{rl}
4=& w(\Gamma\cap(S^2-D_1))\\
= & w(\Gamma\cap{\rm Int}D_2)+w(\Gamma\cap\partial D_2)+w(\Gamma\cap(S^2-(D_1\cup D_2)))\\
\ge & w(\Gamma\cap{\rm Int}D_2)+2+1.
\end{array}$ 
\end{center}
Hence we have $w(\Gamma\cap$Int$D_2)\le1$.
Since the edges in $\partial D_2$ are oriented from $w_3$ to $w_2$,
we have $w(\Gamma\cap$Int$D_2)=0$
by Lemma~\ref{Theorem2AngledDisk} and Step~$1$.
Thus $w_6\in S^2-(D_1\cup D_2)$.

By (2),
we have $w_6\in\Gamma_{m+2\varepsilon}$.
Thus 
$w_6\in\Gamma_{m+\varepsilon}$ or
$w_6\in\Gamma_{m+3\varepsilon}$.
Suppose $w_6\in\Gamma_{m+3\varepsilon}$.
Since $w(\Gamma\cap$Int$D_1)=2$, 
by Lemma~\ref{LoopTwoVertices}(2)
the white vertices in Int$D_1$
are in $\Gamma_{m+\varepsilon}\cap\Gamma_{m+2\varepsilon}$.
Thus $\Gamma_{m+3\varepsilon}$ contains exactly one white vertex $w_6$.
This contradicts Lemma~\ref{ComponentTwoWhite}.
Hence $w_6\in\Gamma_{m+\varepsilon}$.
Thus $w_6\in\Gamma_{m+\varepsilon}\cap\Gamma_{m+2\varepsilon}$.

We can show $w_2,w_3\in\Gamma_{m-\varepsilon}$  or $w_2,w_3\in\Gamma_{m+\varepsilon}$ by IO-Calculation with respect to $\Gamma_{m-\varepsilon}$ in $Cl(S^2-(D_1\cup D_2))$.
If $w_2,w_3\in\Gamma_{m-\varepsilon}$, 
then there exist two lenses of type $(m,m-\varepsilon)$. 
This contradicts Lemma~\ref{NoLens}.
Thus $w_2,w_3\in\Gamma_{m+\varepsilon}$.
This completes the proof of Step~$4$.
\hfill $\square$\\

\begin{figure}[t]
\centerline{\includegraphics{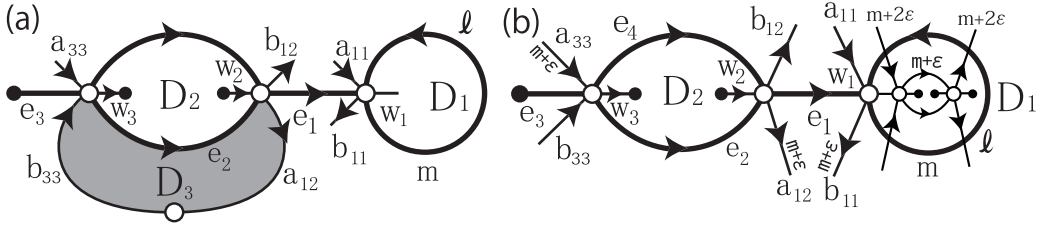}}
\vspace{5mm}
\caption{\label{fig24}The gray region is the disk $D_3$.}
\end{figure}

By Step~$3$, Step~$4$, Lemma~\ref{LoopTwoVertices}(2)
and
Lemma~\ref{Theorem2AngledDisk},
we have the pseudo chart as shown in 
Fig.~\ref{fig24}(b).
From now on,
we use the notations as shown in Fig.~\ref{fig24}(b).\\


{\it Proof of Step~$5$.} 
For the edge $a_{12}$ in Fig.~\ref{fig24}(b),
there are four cases:
(1) $a_{12}=a_{11}$,
(2) $a_{12}=a_{33}$,
(3) $a_{12}=b_{33}$, 
(4) $a_{12}$ contains one of $w_6,w_7$.

{\bf Case (1).}
The edges $a_{12},e_1$ separate the annulus $Cl(S^2-(D_1\cup D_2))$ into two regions.
Each region must contain a white vertex in its interior.
In one of the regions,
 there exists a loop of label $m+\varepsilon$ 
 whose associated disk does not contain any white vertices in its interior.
This contradicts Lemma~\ref{LoopTwoVertices}(1).

{\bf Case (2).}
The edges $a_{12}$ separates the annulus $Cl(S^2-(D_1\cup D_2))$ into two regions.
In a similar way to Case (1),
we can show that
 there exists a loop of label $m+\varepsilon$ 
 whose associated disk does not contain any white vertices in its interior.
This contradicts Lemma~\ref{LoopTwoVertices}(1).

{\bf Case (3).}
The union $a_{12}\cup e_2$ bounds a lens. 
This contradicts Lemma~\ref{NoLens}.

Thus Case (4) occurs.
This completes the proof of Step~$5$.
\hfill$\square$\\


{\it Proof of Step~$6$.}
For the edge $b_{33}$ in Fig.~\ref{fig24}(b),
there are four cases:
(1) $b_{33}=b_{11}$,
(2) $b_{33}=b_{12}$,
(3) $b_{33}\ni w_6$,
(4) $b_{33}\ni w_7$.

{\bf Case (1).}
The edges $b_{11},e_1$ separate the annulus $Cl(S^2-(D_1\cup D_2))$ into two regions.
Let $F$ be one of the two regions containing $w_6\in a_{12}$.
If $w_7\in F$,
then $a_{33}=b_{12}$.
Hence the union $a_{33}\cup e_4$ bounds a lens.
This contradicts Lemma~\ref{NoLens}.
If $w_7\not\in F$, then
there exists a loop of label $m+\varepsilon$ in $F$ whose associated disk does not contain any white vertices in its interior.
This contradicts Lemma~\ref{LoopTwoVertices}(1).

{\bf Case (2).}
The edges $b_{11}$ separate the annulus $Cl(S^2-(D_1\cup D_2))$ into two regions.
Let $F$ be one of the two regions containing $a_{33}$.
Then $F$ must contain $w_7$.
Thus there exists a loop of label $m+\varepsilon$ in $F$
whose associated disk does not contain any white vertices in its interior.
This contradicts Lemma~\ref{LoopTwoVertices}(1).

{\bf Case (3).}
By Lemma~\ref{CorDiskLemma}
the disk $D_3$ bounded by $a_{12}\cup b_{33}\cup e_2$ contains at least one white vertex in its interior as shown in Fig.~\ref{fig24}(a). 
Thus $w_7\in D_3$.
For the white vertex $w_7$,
there are two edges of $\Gamma_{m+\varepsilon}$ not containing middle arcs.
One contains $w_6$, and the other is a loop.
Hence there exists a loop of label $m+\varepsilon$
in $D_3$ whose associated disk does not contain any white vertices in its interior.
This contradicts Lemma~\ref{LoopTwoVertices}(1).

Hence Case (4) occurs.
This completes the proof of Step~$6$.
\hfill$\square$\\


{\it Proof of Step~$7$.}
For the edge $a_{33}$ in Fig.~\ref{fig24}(b),
there are four cases:
(1) $a_{33}=b_{12}$,
(2) $a_{33}=b_{11}$,
(3) $a_{33}\ni w_6$, and
(4) $a_{33}\ni w_7$.

{\bf Case (1).}
The union $a_{33}\cup e_4$ bounds a lens.
This contradicts Lemma~\ref{NoLens}.

{\bf Case (2).}
We have $b_{12}=a_{11}$.
Let $e_7$ be an edge of $\Gamma_{m+\varepsilon}$ with $w_7\in e_7$ different from $b_{33}$ 
not containing a middle arc at $w_7$.
By Assumption~\ref{NoTerminal},
the edge $e_7$ is not a terminal edge.
Thus $e_7$ is a loop or $w_6\in e_7$.
If $e_7$ is a loop,
then the associated disk of $e_7$ does not contain any white vertices in its interior.
This contradicts Lemma~\ref{LoopTwoVertices}(1).
Hence $w_6\in e_7$.

There are four edges of $\Gamma_{m+2\varepsilon}$ intersecting $\partial D_1$,
say $e_1',e_2',e_3',e_4'$.
Since there is no lens by Lemma~\ref{NoLens},
each of the four edges must contain one of $w_6,w_7$.
Thus there are two terminal edges of $\Gamma_{m+\varepsilon}$ in the disk $D_3$ containing $w_6,w_7$ respectively (see Fig.~\ref{fig25}(a)).
However it is impossible that
for each pair $1\le i<j\le 4$
with Int$(e_i')\cap$Int$(e_j')=\emptyset$.

{\bf Case (3).}
There exists an edge $e_7$ of $\Gamma_{m+\varepsilon}$ 
containing $w_7$ different from $b_{33}$ 
not containing a middle arc at $w_7$.
In a similar way to Case (2),
we can show $w_6\in e_7$.
Let $e_7'$ be the edge of $\Gamma_{m+\varepsilon}$ containing $w_7$ different from $b_{33}$ and $e_7$.
Then $e_7'$ is a terminal edge.
However we have a contradiction by IO-Calculation with respect to $\Gamma_{m+2\varepsilon}$ in the disk $D_3$ bounded by $a_{33}\cup b_{33}\cup e_7$ containing $e_3$ 
(see Fig.~\ref{fig25}(b) for the case $D_3\supset e_7'$).  

Hence Case (4) occurs.
This completes the proof of Step~$7$.
\hfill$\square$\\


\begin{figure}[t]
\centerline{\includegraphics{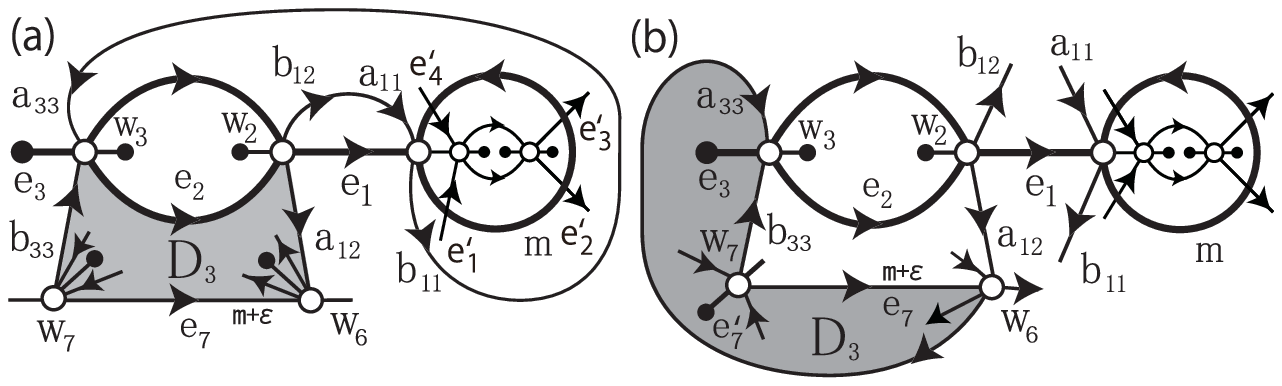}}
\vspace{5mm}
\caption{\label{fig25} The gray region is the disk $D_3$.}
\end{figure}

{\it Proof of Step~$8$.}
So far, we have a pseudo chart as shown in Fig.~\ref{fig26}(a).
For the edge $b_{11}$ in Fig.~\ref{fig26}(a),
there are three cases:
(1) $b_{11}=a_{11}$, 
(2) $b_{11}\ni w_7$, 
(3) $b_{11}\ni w_6$.

{\bf Case (1).}
We have a solar eclipse.
This contradicts Proposition~\ref{NoSolarEclipse}.

{\bf Case (2).}
There exists a loop of label $m+\varepsilon$ containing $w_6$ whose associated disk does not contain any white vertices 
in its interior.
This contradicts Lemma~\ref{LoopTwoVertices}(1).

{\bf Case (3).}
There are two cases: $b_{12}\ni w_7$ or $b_{12}=a_{11}$.
But if $b_{12}\ni w_7$,
then the union $b_{12}\cup a_{33}\cup e_4$ bounds a disk $D'$ with $w(\Gamma\cap$Int$D')=0$.
This contradicts Lemma~\ref{CorDiskLemma}.
Hence $b_{12}=a_{11}$.

Let $e_6$ be the edge of $\Gamma_{m+\varepsilon}$ containing $w_6$ different from $a_{12},b_{11}$.
Let $D_3$ be the disk bounded 
by $a_{12}\cup b_{12}\cup b_{11}$ 
with $e_1\subset D_3$.
Then 
\begin{enumerate}
\item[(*)]$w(\Gamma\cap$Int$D_3)=0$.
\end{enumerate}
If $e_6\subset D_3$,
then we have a contradiction by IO-Calculation with respect to $\Gamma_{m+2\varepsilon}$ in $D_3$.
Thus we have $e_6\not\subset D_3$
and there exists a terminal edge of $\Gamma_{m+2\varepsilon}$ containing $w_6$ in $D_3$.

Now $e_6$ is a terminal edge or $e_6\ni w_7$.
If $e_6$ is a terminal edge,
then in a neighborhood of $D_3$,
the chart $\Gamma$ contains the pseudo chart as shown in Fig.~\ref{fig19}(c).
By Triangle Lemma(Lemma~\ref{LemmaTriangle}(2)), we have $w(\Gamma\cap$Int$D_3)\ge1$.
This contradicts (*).
Hence $e_6$ is not a terminal edge.
Thus $e_6\ni w_7$  (see Fig.~\ref{fig26}(b)).

Let $e_1',e_2',e_3'$ be the edges of $\Gamma_{m+2\varepsilon}$ incident with the $2$-angled disk in $D_1$ in Fig.~\ref{fig26}(b).
For the edge $e_1'$, 
there are four cases:
(3-1) $e_1'=e_2'$,
(3-2) $e_1'=e_3'$,
(3-3) $e_1'=a_{66}$,
(3-4) $e_1'=b_{66}$.

{\bf Case (3-1).}
There exists a lens of type $(m+\varepsilon,m+2\varepsilon)$.
This contradicts Lemma~\ref{NoLens}.

Let $e'$ be the edge of $\Gamma_{m+\varepsilon}$ in $D_1$ containing $w_1$,
and $D_4$ the $2$-angled disk of $\Gamma_{m+\varepsilon}$ in $D_1$.

{\bf Case (3-2).}
The edge $e_1'$ separates the annulus $Cl(S^2-D_4)$ into two regions.
Let $F$ be the one of the two regions containing the edge $e_2'$.
Then $w(\Gamma\cap$Int$F)=0$.
Hence we have a contradiction by IO-Calculation with respect to $\Gamma_{m+2\varepsilon}$ in $F$.

{\bf Case (3-3).}
The simple arc $e_1'\cup e'$ separates the annulus 
$Cl(S^2-D_3)$ into two regions.
Let $F$ be the one of the two regions containing the edge $b_{66}$.
Then Int$F$ does not contain any white vertices except $w_3$ and $w_7$.
Hence
we have a contradiction by IO-Calculation with respect to $\Gamma_{m+2\varepsilon}$ in $F$.

Thus Case (3-4) occurs.
In a neighborhood of $D_3$,
the chart $\Gamma$ contains the pseudo chart as shown in Fig.~\ref{fig19}(a).
By Triangle Lemma(Lemma~\ref{LemmaTriangle}(1)),
we obtain a chart 
containing the pseudo chart as shown in Fig.~\ref{fig19}(b)
(see Fig.~\ref{fig26}(c)).
Apply a C-III move to the chart as shown in Fig.~\ref{fig26}(c),
we obtain a minimal chart as shown in Fig.~\ref{fig26}(d). 
However there exists a lens of type $(m+\varepsilon,m+2\varepsilon)$.
This contradicts Lemma~\ref{NoLens}.
Therefore 
there does not exist a pair of skew eyeglasses of type $1$.
This completes the proof of Step~$8$,
and the proof of Theorem~\ref{NoLoop}.
\hfill$\square$

\begin{figure}[t]
\centerline{\includegraphics{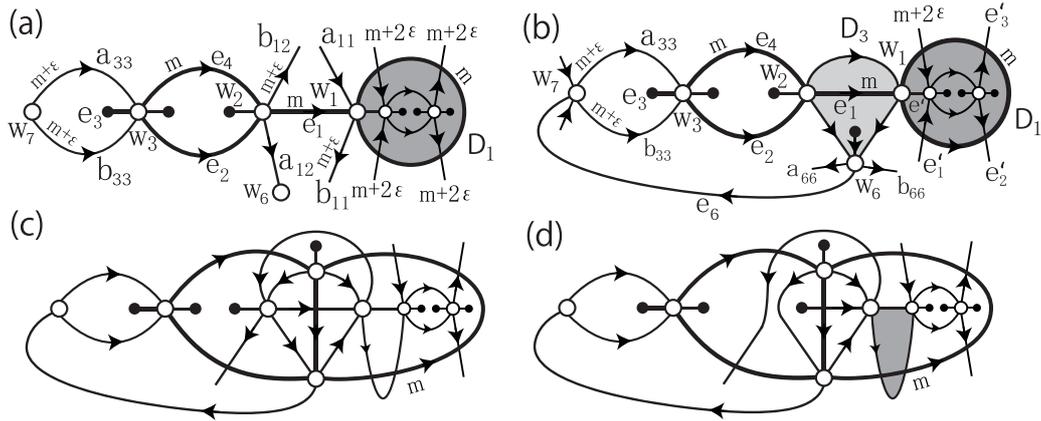}}
\vspace{5mm}
\caption{\label{fig26} The thick lines are edges of label $m$.
(a) The gray region is the disk $D_1$.
(b) The gray regions are the disks $D_1,D_3$.
(d) The gray region is a lens.}
\end{figure}



\vspace{5mm}

\begin{minipage}{65mm}
{Teruo NAGASE
\\
{\small Tokai University \\
4-1-1 Kitakaname, Hiratuka \\
Kanagawa, 259-1292 Japan\\
\\
nagase@keyaki.cc.u-tokai.ac.jp
}}
\end{minipage}
\begin{minipage}{65mm}
{Akiko SHIMA 
\\
{\small Department of Mathematics, 
\\
Tokai University
\\
4-1-1 Kitakaname, Hiratuka \\
Kanagawa, 259-1292 Japan\\
shima@keyaki.cc.u-tokai.ac.jp
}}
\end{minipage}

\vspace{5mm}

{\bf The List of words}\\
\\
{\small $
\begin{array}{ll||ll}
\text{$k$-angled disk of $\Gamma_m$} & p5 &
\text{pair of eyeglasses} & p19 \\

\text{associated disk of a loop} & p4 &
\text{pair of skew eyeglasses} & p19 \\

\text{C-move equivalent} & p3 &
\text{pseudo chart} & p2 \\

\text{C-move keeping $X$ fixed} & p5 &
\text{point at infinity $\infty$} & p4 \\

\text{$(D,\alpha)$-arc free} & p23 &
\text{ring} & p4 \\

\text{edge of $\Gamma_m$} & p3 &
\text{RO-family of a pseudo chart} & p5 \\

\text{feeler} & p6 &
\text{simple hoop} & p4 \\

\text{hoop} & p2, p4 &
\text{special $k$-angled disk} & p6 \\

\text{inward arc} & p9 &
\text{solar eclipse} & p16 \\

\text{IO-Calculation with respect to $\Gamma_k$} & p10 &
\text{terminal edge} & p3 \\

\text{lens} & p7 &
\text{$\Gamma_m$} & p2 \\

\text{locally minimal with respect to a disk} & p6  &
\text{$w(X)$} & p5 \\

\text{loop} & p1, p3 &
\text{$c(X)$} & p5 \\

\text{middle arc} & p2 &
\text{$\alpha[p,q]$} & p8 \\

\text{minimal chart} & p3 &
\text{$a_{ij}$} & p9 \\

\text{outward arc} & p9 &
\text{$b_{ij}$} & p9 \\

\end{array}
$}

\end{document}